\documentclass[11pt]{article}
\usepackage{amsmath,amssymb,amsthm,latexsym}
\usepackage{amsfonts}
\usepackage{amscd}
\usepackage{hyperref}
\usepackage{amsfonts}
\usepackage{amscd}
\usepackage{color,xcolor}
\usepackage{hyperref}
\usepackage{fancyhdr}
\usepackage{indentfirst}
\usepackage{listings}
\usepackage{graphics}
\usepackage{graphicx}
\usepackage{lineno}
\usepackage{geometry}
\usepackage{amsmath, amssymb, graphics, setspace}
\usepackage{color}
\newcommand{\mathsym}[1]{}
\newcommand{\unicode}[1]{}
\newtheorem*{newprob}{Problem}

\theoremstyle{newthe}
\newtheorem*{newthe}{Conjecture  (Montiel-Urbano)}
\newtheorem{mainthe}{Theorem}

\def\q{\theta}

\def\l{\lambda}

\def\v{\nu}

\def\r{\rho}

\def\s{\sigma}

\def\o{\omega}

\def\O{\Omega}
\def\e{\varepsilon}
\def\II{\mathrm{II}}
\def\bq{\overline\theta}
\def\bO{\overline\Omega}
\def\bo{\overline\omega}
\def\be{\overline e}
\def\bn{\overline n}
\def\bps{\overline \psi}
\def\bph{\overline \phi}
\def\bPs{\overline \Psi}
\def\bPh{\overline \Phi}
\def\bH{\overline H}
\def\bv{\overline v}
\def\bep{\overline \epsilon}
\def\br{\overline \rho}
\def\xz{x_{z}}
\def\xbz{x_{\bar z}}
\def\p{\partial}
\def\beq{\begin{equation}}
\def\eeq{\end{equation}}
\def\bspl{\begin{split}}
\def\espl{\end{split}}
\def\psib1{\psi_{\bar1}}
\def\psbb1{\psi_{\bar1\bar1}}
\def\psb11{\psi_{\bar11}}
\def \vb1{v_{\bar1}}
\def \bav1{\overline{v_1}}
\def \bavb1{\overline{v_{\bar1}}}
\def \vba11{v_{\bar1 1}}
\def \v1b1{v_{1\bar1}}
\def \baph1{\bph_{1}}
\def \bbaph11{\bph_{11}}

\newcommand{\ca}{\overset{\circ}{{\mathcal A}}}
\newcommand{\bl}{\big\langle}
\newcommand{\Br}{\big\rangle}
\newcommand{\bll}{\big\langle\!\!\big\langle}
\newcommand{\Brr}{\big\rangle\!\!\big\rangle}
\newcommand{\B}{B(e,e)}
\newcommand{\Bb}{B(\bar e,\bar e)}
\DeclareMathOperator{\He}{Hess}
\newcommand{\n}{\nabla_{e}^\perp}
\newcommand{\nb}{\nabla_{\bar e}^\perp}
\newcommand{\ric}{\overline{Ric}}

\newcommand{\Be}{\bar e}
\newcommand{\vH}{\vec{H}}
\newcommand{\bvaa}{\bar{v}_{11}}

\newcommand{\vaa}{v_{11}}
\newcommand{\vbb}{v_{\bar{1}\bar{1}}}
\newcommand{\vab}{v_{1\bar1}}
\newcommand{\vBa}{v_{\bar1 1}}
\newcommand{\vaba}{v_{1\bar1 1}}
\newcommand{\vaab}{v_{11\bar1}}
\newcommand{\vbba}{v_{\bar1\bar1 1}}
\newcommand{\vbab}{v_{\bar1 1\bar1}}

\newcommand{\Ga}{\bar{\phi}_{1}}

\newcommand{\G}{\bar{\phi}}
\newcommand{\bF}{\bar{\psi}}

\newtheorem{theorem}{Theorem}[section]
\newtheorem{lemma}[theorem]{Lemma}
\newtheorem{proposition}[theorem]{Proposition}
\newtheorem{corollary}[theorem]{Corollary}

\theoremstyle{definition}
\newtheorem{definition}[theorem]{Definition}
\newtheorem{remark}[theorem]{Remark}

\title{\textbf{Willmore surfaces in 4-dimensional conformal manifolds}}
\author {Changping Wang~~~~~~Zhenxiao Xie\\}
\date{}

\begin{document}
\maketitle

\begin{abstract}
This paper is dedicated to the exploration of the conformal Willmore functional for surfaces within 4-dimensional conformal manifolds. We provide a detailed calculation of both the first and second variations, and present the Euler-Lagrange equation of this functional in a conformally invariant form. Utilizing the second variation formula we derived, we demonstrate that the Clifford torus in \(\mathbb{C}P^2\) is strictly Willmore-stable. This finding strongly supports the conjecture proposed by Montiel and Urbano [J. reine angew. Math. 546 2002, 139-154], which posits that the Clifford torus in \(\mathbb{C}P^2\) minimizes the Willmore functional among all tori. Moreover, by applying our formula to complex curves in \(\mathbb{C}P^2\), we establish that the first nonzero eigenvalue of the Jacobi operator is at least 12. In the context of 4-dimensional locally symmetric spaces, we construct several holomorphic differentials to show that among all minimal 2-spheres, only those super-minimal ones can be Willmore. 
\end{abstract}



\par\medskip\noindent
\section{Introduction}\label{sec-intro}
For closed surfaces in a given Riemannian manifold, considering the trace-free part of the second fundamental form, to integrate the norm square of it, one can obtain a global conformal invariant, which is known as  the conformally invariant Willmore functional (called {\em Willmore functional} for short in this paper, denoted by $\mathcal{W}$). 

In conformal spheres 
(the conformal compactification of real space forms), up to a topological invariant, 
the Willmore functional is also known as the integration of the norm square of mean curvature vector.   Since the work of Blaschke \cite{Blaschke} and Thomsen \cite{Thomsen}, this functional has been investigated in depth (see \cite{Willmore, Chen, Weiner, Li-Yau, Bryant, Pinkall, Ejiri, Kusner, Minicozzi, Guo-Wang-Li, Kuwert, Bernard-Riviere, Marques-Neves2, Laurain-Riviere, Riviere} and references therein). Many of these  studies concentrated on solving the famous Willmore conjecture, which is proposed by Willmore in \cite{Willmore}.  He conjectured that the Clifford torus minimizes the Willmore functional among all tori  in $\mathbb{S}^3$. Now this conjecture has been proved by Marques and Neves in \cite{Marques-Neves}. There are also several other works concerning the classification of Willmore surfaces in conformal spheres, such as the classification of Willmore $2$-spheres \cite{Bryant, Ejiri, Ma-Wang-Wang}.  

For surfaces in curved Riemannian manifolds (other than real space forms), the Willmore functional has also been studied. In \cite{Pedit-Willmore}, the authors stated the Euler-Lagrangian equation for general submanifolds without writing down a proof. Hu and Li rediscovered  this equation in  \cite{Hu-Li}, where Willmore complex submanifolds in a complex space form were also studied. In \cite{Mondino2},  Mondino and Rivi\`ere rewrote the Euler-Lagrange equation in a conservative form which makes sense for weak immersions. We point out that the Willmore functional was referred to the conformal Willmore functional in that paper, where the $L^2$-norm square of 
the mean curvature vector was called Willmore functional (we call it the {\em non-conformal Willmore functional} in this paper).  A rigidity theorem was obtained by them, which states that a compact $3$-dimensional Riemannian manifold of constant scalar curvature has constant sectional curvature, if and only if, every smooth constant mean curvature sphere is Willmore. In \cite{Mondino1}, the perturbative method was employed by Mondino to show the existence of Willmore spheres in $(\mathbb{R}^3, g_\epsilon)$, where  $ g_\epsilon$ is a metric close and asymptotic to the Euclidean one. Kuwert, Mondino and Schygulla \cite{KMS} established the existence of immersed spheres that minimize curvature functionals in compact Riemannian manifolds $(M^3, g)$ under certain curvature conditions.  
In \cite{Ikoma}, Ikoma, Malchiodi and Mondino constructed smooth embedded Willmore tori with small area constraint in Riemannian three-manifolds, under some curvature condition. 
The algebraic structure of global conformal invariants for submanifolds has been investigated in \cite{Mondino3}. Recently, under some uniformly bounded conditions, Michelat and Mondino \cite{Michelat-Mondino}  obtained the quantization of energy for Willmore spheres in general closed Riemannian manifolds. This was thought as 
the first breakthrough on the generalization of Rivi\'ere’s min-max theory developed for Willmore spheres in Euclidean spaces \cite{Riviere}. 

In \cite{Montiel-Urbano}, the Willmore functional $\mathcal{W}$ was studied by Montiel and Urbano for closed oriented surfaces in oriented $4$-dimensional Riemannian manifolds from the twistor viewpoint. They decomposed this functional as two parts, $\mathcal{W}^+$ and $\mathcal{W}^-$, which are both conformally invariant individually, and closely related to the Penrose twistor bundles over $4$-manifolds.  It was shown that super-conformal surfaces (also known as twistor holomorphic surfaces) with positive (resp. negative) spin minimize $\mathcal{W}^+$ (resp. $\mathcal{W}^-$).  Such kind of examples can be constructed easily when the ambient space is anti-self dual (resp. self dual), for which the twistor bundle is a complex manifold. 
In that paper, 
$\mathcal{W}^-$  was well studied for the case of $\mathbb{C}P^2$, which is self dual but not anti-self dual. The authors showed that not only super-minimal surfaces with negative spin but also general minimal surfaces are critical points of  $\mathcal{W}^-$. But for the functional $\mathcal{W}^+$, they proved that to be a critical point,  a minimal surface must be super-minimal with positive spin. Then it follows from the classification result of Gauduchon on super-minimal surfaces with positive spin that among all minimal surfaces only complex curves and minimal Lagrangian surfaces are Willmore. Note that in general these examples are not of negative spin, so one do not know whether they minimize the Willmore functional $\mathcal{W}$ and the functional $\mathcal{W}^-$. 
Some explicit examples were studied in \cite{Montiel-Urbano}, including the Clifford torus $T$ 
defined as follows: 
\beq \label{eq-T2}
T=\Big\{[z_1, z_2, z_3]\in \mathbb{C}P^2\, \Big|\, |z_1|=|z_2|=|z_3| \Big\}.
\eeq
Inspired by the 
work of Minicozzi \cite{Minicozzi} on the Willmore conjecture in $\mathbb{R}^4=\mathbb{C}^2$, 
Montiel and Urbano proposed the following conjecture of Willmore type in \cite{Montiel-Urbano}. 
\begin{newthe}
The Clifford torus achieves the minimum of the functional $\mathcal{W}^-$ (and then $\mathcal{W}$) either amongst all tori in $\mathbb{C}P^2$, or amongst all Lagrangian tori in $\mathbb{C}P^2$. 
\end{newthe} 
This conjecture has been verified on certain families of Hamiltonian-minimal Lagrangian tori, see \cite{Ma-Mironov-Zuo} and \cite{Kazhymurat}. 

In this paper, we put the study of Willmore functional for surfaces in $4$-dimensional Riemannian manifolds under the framework of conformal geometry, through which the complex technique is fully used.   

For surfaces in $4$-dimensional conformal manifolds, several conformal invariants 
are constructed by us. With them, we can rewrite the functionals $\mathcal{W}^+$ and $\mathcal{W}^-$ of Montiel and Urbano as two integrals of conformal invariants, from which the minimizing characterization of super-minimal surfaces with positive or negative spin follows directly. 
The first variation of Willmore functional is calculated using the moving frame method. We write down the Euler-Lagrange equation by conformal invariants, see \eqref{eq-Eul} and \eqref{eq-Glcon}. 

Examples of Willmore surfaces in $4$-dimensional locally symmetric spaces 
are discussed. 
We prove that super-minimal surfaces with positive (resp. negative) spin  in locally symmetric, self-dual (resp. anti-self-dual) Einstein $4$-manifolds are Willmore. This can be seen as a generalization of one aforementioned result 
of Montiel and Urbano obtained in the case of $\mathbb{C}P^2$. Due to the work of Derdzi{\'n}ski \cite{Derdzinski}, such kind of ambient spaces include orientable, locally irreducible, and locally symmetric $4$-manifolds,  as well as self-dual K\"ahler-Einstein surfaces. 
Conversely, by constructing several holomorphic differentials, we obtain the following theorem. 
\begin{mainthe} \label{main1}
In a $4$-dimensional locally symmetric space, to be Willmore, a minimal $2$-sphere must be super-minimal. 
\end{mainthe}

The second variation formula serves as a fundamental tool in the study of the Willmore functional. 
The case of conformal spherical ambient spaces has been thoroughly explored by many geometers. In \cite{Weiner}, Weiner derived this formula specifically for minimal surfaces in \(\mathbb{S}^3\). Subsequently, utilizing the conformal Gauss map, Palmer calculated the second variation for general Willmore surfaces in \(\mathbb{S}^3\) \cite{Palmer1, Palmer2}. 
In \cite{Guo-Wang-Li}, Guo, Li and the first author of this paper presented the second variation formula using M\"obius invariants for umbilic-free Willmore submanifolds in \(\mathbb{S}^n\). Additionally, the second variation of the non-conformal Willmore functional has been computed for surfaces in Riemannian manifolds, see \cite{Lamm-Metzger-Schulze} and \cite{Michelat}.

In this paper, we focus on computing the second variation of Willmore surfaces within general 4-dimensional conformal manifolds. The formula is detailed in Theorem~\ref{thm-2ndvar}. {As an application, we obtain the following results in $\mathbb{C}P^2$, providing strong support to the conjecture by Montiel and Urbano mentioned earlier. 
\begin{mainthe}\label{main2}
The Clifford torus in $\mathbb{C}P^2$ is strictly Willmore-stable, as well as strictly stable for $\mathcal{W}^-$. 
\end{mainthe}
}

{ Comparing with Theorem~\ref{main2}, it is natural to consider the Willmore stability of complex curves. 
In fact, we can arrive at an even stronger conclusion directly, i.e., {\em every complex curve minimizes the Willmore functional $\mathcal{W}$ within its homotopy class}. 
 It is already known that  every complex curve is the minimizer of $\mathcal{W}^+$ and the area functional $\mathcal{A}$ within its homotopy class. Regarding $\mathcal{W}^-$, 
we only need to observe that 
$$\mathcal{W}^-(M^2)=\int_{M^2} (|\vec{H}|^2+2)dA\geq \int_{M^2} 2dA=2\mathcal{A}(M^2),$$
and at the right-hand side, a complex curve achieves the minimum within its homotopy class, which is precisely equal to the value of $\mathcal{W}^-$. 

In this scenario, we can utilize our second variational formula in reverse. By combining it with the stability of \(\mathcal{W}^+\), we show in Remark~\ref{rk-complex} that {\em for complex curves in $\mathbb{C}P^2$, the first non-zero eigenvalue of the Jacobi operator is at least 12}. 
}  

 We organize this paper as follows. Section~\ref{sec-con-inv} discusses the conformal geometry of surfaces in $4$-dimensional conformal manifolds. The first variation of the Willmore functional is calculated in Section~\ref{1st-variation}, where examples of Willmore surfaces in some $4$-dimensional locally symmetric spaces are also discused. In Section~\ref{sec-2sphere}, we prove Theorem~\ref{main1} by constructing several holomorphic differentials for minimal surfaces in $4$-dimensional locally symmetric spaces.  Section~\ref{2nd-variation} is devoted to presenting the calculation of the second variation formula for Willmore surfaces. As an application, we study the Willmore stability of the Clifford torus 
 in Section~\ref{sec-torus}, 
 where the proof of Theorem~\ref{main2} 
 can be found. 
 
\section{Conformal geometry of surfaces in conformal 4-manifold}\label{sec-con-inv}
\par\medskip
Let $(N^4, [h])$ be a $4$-dimensional conformal manifold, where $h$ is a Riemannian metric on $N^4$, and $[h]$ is the conformal class of $h$. Suppose $x: M^2\to (N^4,[h])$ is an immersed surface. In this section, we define some conformal invariants 
of $x(M^2)$.  
\subsection{Basic isometric invariants on $(M^2, x^*h)$}
Let $\{e_1,e_2, e_3, e_4\}$ be a local orthonormal frame for $h$ defined along $x(M^2)$ with the dual coframe as $\{\o_1,\o_2, \o_3, \o_4\}$, such that on $x(M^2)$, $e_1, e_2$ (resp. $e_3, e_4$) are tangent (resp. normal) vector fields. The Levi-Civita connection forms of $(N^4, h)$ are denoted by $\o_{ij}$. They satisfy $\o_{ij}+\o_{ji}=0$ and the following structure equations, 
\begin{equation}\label{eq-structure}
\begin{split}
&~~~~~~~~~~~~~~~~~~~~~~~~d\o_i=\sum_{j=1}^4\o_{ij}\wedge\o_j, ~~~1\leq i\leq 4;\\
&d\o_{ij}-\sum_{k=1}^{4}\o_{ik}\wedge\o_{kj}+\sum_{1\leq k<l\leq4}K_{ijkl}\o_k\wedge\o_l=0,~~~1\leq i,j\leq 4,
\end{split}
\end{equation}
where $K_{ijkl}$ is the curvature tensor of $(N^4, h)$.

In order to analyze the conformal geometry of $x$ more efficiently, we expoit the complex notation. On $(N^4, h)$, set   
\begin{equation}\label{eq-complex}
\begin{split}
e\triangleq e_1+ie_2,~~~n\triangleq e_3+ie_4,~~~\o\triangleq \o_1+i\o_2,~~~\O\triangleq \o_3+i\o_4,~~~~~~\\
\mathrm{II}_1\triangleq(\o_{13}+i\o_{23})+i(\o_{14}+i\o_{24}),~~~\mathrm{II}_2\triangleq(\o_{13}-i\o_{23})+i(\o_{14}-i\o_{24}).
\end{split}
\end{equation}
Then the structure equation \eqref{eq-structure} can be rewritten as below
\begin{align}
d\o=-i\o_{12}\wedge\o+\frac{1}{2}(\mathrm{II}_1\wedge\overline\O+\overline{\II}_2\wedge\O),\label{eq-do}\\
d\O=-i\o_{34}\wedge\O-\frac{1}{2}(\mathrm{II}_1\wedge\overline\o+\II_2\wedge\o),\label{eq-dO}
\end{align}
\begin{equation}\label{eq-dII1}
\begin{split}
d\,\II_1=-i(\o_{12}+\o_{34})\wedge \II_1+\frac{1}{4}\big(&K(e,n,\be,e)\bo\wedge\o+K(e,n,\bn,n)\overline\O\wedge\O-K(e,n,\be,n)\o\wedge\bO\\
&-K(e,n,\be,\bn)\o\wedge\O-K(e,n,e,n)\bo\wedge\bO-K(e,n,e,\bn)\bo\wedge\O\big),
\end{split}
\end{equation}
\begin{equation}\label{eq-dII2}
\begin{split}
d\,\II_2=i(\o_{12}-\o_{34})\wedge\II_2+\frac{1}{4}\big(&K(\bar e,n,\bar e,e)\bo\wedge\o+K(\bar e,n,\bar n,n)\overline\O\wedge\O-K(\bar e,n,\bar e,n)\o\wedge\bO\\
&-K(\be,n,\be,\bn)\o\wedge\O-K(\be,n,e,n)\bo\wedge\bO-K(\be,n,e,\bn)\bo\wedge\O\big).
\end{split}
\end{equation}

Note that $x^*(\O)=0$. Set 
$$\q\triangleq x^*(\o),~~~\q_{12}\triangleq x^*(\o_{12}),~~~\q_{34}\triangleq x^*(\o_{34}).$$ 
It follows from \eqref{eq-dO} that we can assume
\begin{equation}\label{eq-II}
\II_1=H\q+2\,\bps\,\bq,~~~~~~\II_2=H\bq+2\,\phi\,\q, 
\end{equation}
where $H, \phi, \psi$ are three locally defined complex valued functions on $M^2$. By direct calculation, we have
\begin{equation}\label{eq-stru}
\begin{split}
\nabla_e \,e=-i\theta_{12}(e)e+2\bps\,\bn+2\bph\, n,~~~\nabla_{\be}\, e=-i\theta_{12}(\be)e+H\,\bn+\bH\, n,\\
\nabla_e \,n=-i\theta_{34}(e)n-2\bps\,\be-H\, e,~~~\nabla_{\be}\, n=-i\theta_{34}(\be)n-H\,\be-2\phi\, e,
\end{split}
\end{equation}
where $\nabla$ is the Levi-Civita connection of $h$.
It follows that the mean curvature vector 
$$\vec{H}= \frac{(\nabla_{\be} \,e)^\perp}{2}=\frac{H\,\bn+\bH\, n}{2},$$
with $|H|$ being the mean curvature of $x$. Here $\perp$ denotes the orthogonal projection on the normal plane of $x$.   
If we denote the second fundamental form of $x$ by  
$$B=\sum_{1\leq i,j\leq 2}(b_{ij}^3 e_3+b_{ij}^4 e_4)\q_i\otimes \q_j,$$
 then  $B(e,e)=(\nabla_e \,e)^\perp=2\bps\,\bn+2\bph\, n$, and one can verify that 
 \beq\label{eq-2ndform}
 2\phi=\left(b_{12}^4+\frac{b_{11}^3-b_{22}^3}{2}\right)-i \left(b_{12}^3-\frac{b_{11}^4-b_{22}^4}{2}\right), ~~~2\psi=\left(-b_{12}^4+\frac{b_{11}^3-b_{22}^3}{2}\right)-i \left(b_{12}^3+\frac{b_{11}^4-b_{22}^4}{2}\right).
 \eeq
Consequently,  
the norm square of the trace-free part of the second fundamental form can be written as below:  
 \beq \label{eq-tracefree}
 |\overset{\circ} B|^2=4(|\phi|^2+|\psi|^2).
 \eeq
 
It is easy to see that $B(\be,\be)\q^2$ is a globally defined normal-bundle-valued $2$-form, and 
$$\bl B(\be, \be), B(\be,\be)\Br \q^4=8\phi\psi\q^4$$ is a globally defined quartic form on $M^2$, where $\langle ~,~\rangle$ denotes the bilinear inner product on $TN^4\otimes \mathbb{C}$ induced by the Riemannian metric $h$.  In fact, they are both conformal invariants (see Subsection~\ref{subsec-confinv}). 

Using \eqref{eq-2ndform}, one can verify directly that the quartic form $\phi\psi\q^4$ 
vanishes at a given point $p\in M^2$ if and only if the curvature ellipse (see \cite{Guadalupe-Rodriguez, Friedrich} for the precise definition) of $x$ defined by 
$$\vec{H}+\cos(2\varphi)\frac{B(e_1,e_1)-B(e_2,e_2)}{2}+\sin(2\varphi)B(e_1, e_2),~~~0\leq \varphi\leq 2\pi,$$
is a circle at $p$. 
\begin{definition}
Let $x: M^2\rightarrow (N^4, h)$ be a surface. If the quartic form $\phi\psi\q^4$ vanishes identically on $M^2$, then $x$ is called a super-conformal surface. If in addition $\vec{H}=0$, we call $x$ a super-minimal surface.  
\end{definition}
\begin{remark}\label{rk-orientaion}
Suppose $(N^4, h)$ is an oriented 4-dim manifold, $x: M^2\rightarrow (N^4, h)$ is an oriented super-minimal surface, and the orthonormal frame $\{e_1, e_2, e_3, e_4\}$ is chosen such that 
$e_1\wedge e_2\wedge e_3\wedge e_4$ and $e_1\wedge e_2$ coincide the orientation of $N^4$ and $x(M^2)$, respectively. It is straightforward to verify that  the super-minimal surface $x$ of positive (resp. negative) spin, in the sense of \cite{Montiel-Urbano}, can be characterized by $\psi=0$ (resp. $\phi=0$). 
\end{remark}

The following well-known cross sections of the bundle $\text{Hom}(NM^2, NM^2)$ appearing in the second variation of minimal surfaces will be used in the sequel, 
\beq\label{eq-calA}
\ca(\xi)\triangleq\frac{1}{4}\Big(\bl\B,\xi\Br\Bb+\bl\Bb,\xi\Br\B\Big),~~~\ric(\xi)\triangleq\frac{1}{2}\big(K(e,\xi)\Be+K(\Be,\xi)e\big)^\perp, 
\eeq
where $\xi=(a\bn+\bar a n)/{2}$ is a normal vector field. 
It is straightforward to verify that 
\begin{align}
&\bl \ca(\xi),\xi\Br=2\Big(|a|^2\left(|\phi|^2+|\psi|^2\right)+a^2\G\psi+\bar{a}^2\phi \bar \psi \Big). \label{eq-avv}
\end{align}

\subsection{Some formulas on complex covariant derivatives}\label{subsec-cov}
We define the complex covariant derivatives and the covariant differential of $\phi,\, \psi,\, H$ as below: 
\begin{align}
&\phi_1\q+\phi_{\bar 1}\bq=D\phi\triangleq d\phi-i(2\q_{12}-\q_{34})\phi,~~~&\psi_1\q+\psi_{\bar 1}\bq=D\psi\triangleq d\psi-i(2\q_{12}+\q_{34})\psi,\\
&\phi_{11}\q+\phi_{1\bar 1}\bq=D\phi_1\triangleq d\phi_{1}-i(3\q_{12}-\q_{34})\phi_1,~~~&\psi_{11}\q+\psi_{1\bar 1}\bq=D\psi_1\triangleq d\psi_{1}-i(3\q_{12}+\q_{34})\psi_1\\
&\phi_{\bar1 1}\q+\phi_{\bar1 \bar1}\bq=D\phi_{\bar1}\triangleq  d\phi_{\bar1}-i(\q_{12}-\q_{34})\phi_{\bar1},~~~&\psi_{\bar1 1}\q+\psi_{\bar1 \bar1}\bq=D\psi_{\bar1}\triangleq d\psi_{\bar1}-i(\q_{12}+\q_{34})\psi_{\bar1},\label{eq-2ndpsi}\\
&H_1\q+H_{\bar 1}\bq=DH\triangleq dH+i\theta_{34}H,~~~&\bH_1\q+\bH_{\bar 1}\bq=D\bH\triangleq d\bH-i\theta_{34}\bH,\\
&H_{11}\q+H_{1\bar 1}\bq=DH_1\triangleq dH_1-i(\q_{12}-\theta_{34})H_1,~~~&\bH_{11}\q+\bH_{1\bar1}\bq=D\bH_{1}\triangleq  d\bH_1-i(\q_{12}+\theta_{34})\bH_1.
\end{align}
It is straightforward to verify that 
\begin{align}
&\nb \B=4(\bF_1\bn+\Ga n),\\
&{(\nabla^{\perp})}^2 B(e,e,\be,\be)=\nb (\nb B)(e,e)-\nabla_{\nabla_{\be} \be}^\perp B(e,e)= 8(\bF_{11}\bn+\bph_{11} n), \\
&\triangle^\perp \vec H=2(H_{1\bar 1}\bn+\bH_{1\bar 1}n)+i R_{e_2 e_1}^\perp\vec H=2(H_{1\bar 1}\bn+\bH_{\bar 1 1}n)+\frac{R_{1234}}{2}\vec H,  \label{eq-lapH}
\end{align}
where  $\nabla^{\perp}$ is the normal connection of $x$, $\triangle^\perp$ is the normal Laplacian, $R^\perp_{{\,\cdot \,\cdot }}\cdot$
is the normal  curvature tensor, and $R_{1234}$ is the normal Gauss curvature of $x$. We will denote by $R_{1212}$ the Gauss curvature of $x$.  


The Gauss-Codazzi-Ricci equations of $x$ are given by
\begin{equation}\label{eq-Gauss}
R_{1212}=K_{1212}+|H|^2-2(|\phi|^2+|\psi|^2);
\end{equation}
\begin{equation}\label{eq-Codazzi}
2\phi_{\bar1}-H_1=\frac{1}{4}K(\be,n,\be,e),~~~~~~2\psi_{\bar1}-\overline H_1=\frac{1}{4}K(\be,\bn,\be,e);
\end{equation}
\begin{equation}\label{eq-Ricci}
R_{1234}=K_{1234}+2(|\phi|^2-|\psi|^2).
\end{equation}
Using \eqref{eq-Codazzi}, we can obtain 
\begin{equation}\label{eq-sderi}
\begin{split}
8(2\phi_{\bar1\bar1}-H_{1\bar1})=&(\nabla_e\,K)(\be,n,\be,e)+H\Big(K(\be,e,\bn,n)+K(\be,n,\bn,e)-K(\be,e,\be,e)\Big)\\
&+\bH K(\be,n,n,e)+2\Big(K(\be,n,\be,\bn)\bps+K(\be,n,\be,n)\bph\Big),\\
8(2\psi_{\bar1\bar1}-\bH_{1\bar1})=&(\nabla_e\,K)(\be,\bn,\be,e)+\bH\Big(K(\be,e,n,\bn)+K(\be,\bn,n,e)-K(\be,e,\be,e)\Big)\\
&+H K(\be,\bn,\bn,e)+2\Big(K(\be,\bn,\be,\bn)\bps+K(\be,\bn,\be,n)\bph\Big),
\end{split}
\end{equation}
which relates ${(\nabla^{\perp})}^2 B(e,e,\be,\be)$ to $\triangle^\perp \vec H$. 

Similarly, for a general normal vector field $\xi=(a\bn+\bar a n)/{2}$, we define 
\beq \label{eq-cov}
a_1\q+a_{\bar 1}\bq=Da\triangleq da+i\theta_{34}a, ~~~\bar a_1\q+\bar a_{\bar 1}\bq=D\bar a\triangleq d\bar a-i\theta_{34}\bar a,
\eeq
\beq \label{eq-conv2}
a_{11}\q+a_{1\bar 1}\bq=Da_1\triangleq da_1-i(\q_{12}-\theta_{34})a_1,~~~ a_{\bar11}\q+a_{\bar1\bar1}\bq=Da_{\bar1}\triangleq  da_{\bar1}+i(\q_{12}+\theta_{34}) a_{\bar{1}}.
\eeq 
Then it is easy to verify that 
\begin{align}
&\nb\xi =a_1 \bn+\bar a_1 n,\label{eq-nbev}\\
&\He \xi(\Be,\Be)=2\left(a_{11} \bn+\bar{a}_{11} n\right),\label{eq-hessV}\\
&\He \xi(\Be,e)=2\left(a_{1\bar 1} \bn+\bar a_{1\bar 1} n\right)=\Delta^\perp \xi-\frac{1}{2}R^\perp_{e \Be}\xi, \label{eq-hessV2}
\end{align}
where $\He \xi$ is the normal Hessian of $\xi$ defined as below:  
$$\He \xi(X,Y)\triangleq \nabla^\perp_Y\nabla^\perp_X \xi-\nabla_{\nabla_{Y}X} ^\perp \xi.$$

We will make use of the following lemma, which involves integration by parts, in the subsequent variational analysis of the Willmore functional.
\begin{lemma}\label{lem-int1}
Suppose $M^2$ is a closed surface. Let $f\q^{r_1}\bar{\theta}^{s_1}$ and $\tilde{f}\q^{r_2}\bar{\theta}^{s_2}$ be two tensors. 

(1) If $r_1+r_2=s_1+s_2+1$, then 
$$\int_{M^2}(Df)\wedge \tilde{f}\q=-\int_{M^2}(D\tilde{f})\wedge f\q,~~~ {\rm i.e.},~~~\int_{M^2}f_{\bar1}\tilde{f} dA_x=-\int_{M^2}f\tilde{f}_{\bar1}dA_x.$$

(2) If $r_1+r_2=s_1+s_2-1$, then
$$\int_{M^2}(Df)\wedge \tilde{f}\bq=-\int_{M^2}(D\tilde{f})\wedge f\bq, ~~~ {\rm i.e.},~~~ \int_{M^2}f_{1}\tilde{f} dA_x=-\int_{M^2}f\tilde{f}_{1}dA_x.$$
\end{lemma}
\begin{proof} Note that if $r_1+r_2=s_1+s_2+1$ (resp. $r_1+r_2=s_1+s_2-1$), then $f \tilde{f} \q$ (resp. $f\!\tilde{f} \bq$) is a globally defined $1$-form, and 
$$(Df)\wedge \tilde{f}\q+(D\tilde{f})\wedge f\q=d(f\tilde{f}\q),~~~\left({\rm resp.}~ (Df)\wedge \tilde{f}\q+(D\tilde{f})\wedge f\q=d(f\tilde{f}\bq)\right),$$
from which, the conclusion follows. 
\end{proof}

At the end of this subsection, we show the expressions of some invariants under a local complex coordinate. Let $z=u+iv$ be a local complex coordinate on $(M^2,x^*h)$ such that
$$ \theta=e^\s dz,~~~e=2e^{-\s}x_{\bar z},$$
where $\sigma$ is a local defined real-valued function on $M^2$. Define
$$\Phi=h(\nabla_{x_z}x_z, n), ~~~\Psi=h(\nabla_{x_z}x_z, \bn),$$
then it is easy to see that $\phi=e^{-2\s}\Phi,~\psi=e^{-2\s}\Psi$, and $\Phi\Psi dz^4$ are globally defined quartic differential on $M^2$. In the local complex
coordinate $z$, we have 
\begin{gather*}
\theta_{12}=i(\s_{\bar z}d\bar z-\s_zdz),~~~~~~\theta_{34}=i(\r dz- \bar \rho d\bar z),~~~~~~H_1=e^{-\s}(H_z-\r H),\\
H_{\bar1}=e^{-\s}(H_{\bar z}+\bar \rho H),~~~~~~\phi_{\bar1}=e^{-3\s}(\Phi_{\bar z}+\bar{\rho}\Phi),~~~~~~\psi_{\bar1}=e^{-3\s}(\Psi_{\bar z}-\bar{\rho}\Psi),\\
\phi_{\bar1\bar1}=e^{-4\s}\big(\Phi_{\bar z \bar z}-2(\s_{\bar z}-\bar{\rho})\Phi_{\bar z}+(\bar{\rho}^2-2\s_{\bar z}\bar{\rho}+\bar{\rho}_{\bar z})\Phi\big),~
\psi_{\bar1\bar1}=e^{-4\s}\big(\Psi_{\bar z \bar z}-2(\s_{\bar z}+\bar{\rho})\Psi_{\bar z}+(\bar{\rho}^2+2\s_{\bar z}\bar{\rho}-\bar{\rho}_{\bar z})\Psi\big),
\end{gather*}
where $\rho=h(\nabla_{x_z}n, \bn)/2$.
The Codazzi equation \eqref{eq-Codazzi} can be rewritten as
\begin{equation}\label{eq-Coda-complex}
\Phi_{\bar z}+\bar{\rho}\Phi=\frac{e^{2\s}}{2}(H_z-\r H)-K(x_{z},n,x_{\bar z},x_z),~~~~~~\Psi_{\bar z}-\bar{\rho}\Psi=\frac{e^{2\s}}{2}(\bH_z+\r \bH)-K(x_{z},\bn, x_{\bar z},x_z).
\end{equation}
\subsection{Some conformal invariants of $x(M^2)$}\label{subsec-confinv}
Let $\widetilde h=e^{2\lambda}h\in[h]$ be another metric on $N^4$. We make a convention that the geometric object with respect to $\widetilde h$ is denoted by the same symbol as for $h$ but with a tilde on it. It is easy to check that
$$\widetilde\o=e^{\l}\o,~~~\widetilde\O=e^{\l}\O,~~~\widetilde{\o_{12}}=\o_{12}+\lambda_1\o_2-\lambda_2\o_1,~~~\widetilde{\o_{34}}=\o_{34}+\lambda_3\o_4-\lambda_4\o_3,$$
$$\widetilde{\II_1}=\II_1+(\lambda_1+i\lambda_2)\O-(\lambda_3+i\lambda_4)\o,~~~\widetilde{\II_2}=\II_2+(\lambda_1-i\lambda_2)\O-(\lambda_3+i\lambda_4)\bo,$$
where $\l_j=e_j(\l)$. The following conformal change of curvature tensor is well-known,    
\begin{equation}\label{eq-curv}
\begin{split}
e^{-2\lambda}\widetilde{K}(X,Y,Z,W)=&K(X,Y,Z,W)-|\nabla \l|^2\Big(h(X,Z)h(Y,W)-h(X,W)h(Y,Z)\Big)\\
&-h(Y,W)\Big(D^2\lambda(X,Z)-X(\lambda)Z(\lambda)\Big)-h(X,Z)\Big(D^2\lambda(Y,W)-Y(\l)W(\l)\Big)\\
&+h(X,W)\Big(D^2\l(Y,Z)-Y(\l)Z(\l)\Big)+h(Y,Z)\Big(D^2\l(X,W)-X(\l)W(\l)\Big),
\end{split}
\end{equation}
where $D^2\lambda$ is the Hessian of $\lambda$ with respect to $h$. 
Applying these transformation formulas to the surface $x$, we have immediately 
\begin{equation}\label{eq-Hphs}
\widetilde \q=e^{\lambda}\q,~~~e^\l\widetilde H=H-n(\lambda),~~~e^\l\widetilde \phi=\phi,~~~e^\l\widetilde\psi=\psi,~~~\widetilde{B}(e,e)= B(e,e). 
\end{equation}
\begin{equation}\label{eq-Ric-con}
\widetilde{\text{Ric}}(e,e)=\text{Ric}(e,e)-2\big(D^2\l(e,e)-e(\l)e(\l)\big). 
\end{equation}
\begin{definition}
We call zeros of $\phi$ (resp. $\psi$) isotropic (resp. anti-isotropic) points. 
\end{definition}
Note that both isotropic points and anti-isotropic points are conformal invariants. The umbilical points are characterized as the common zeros of $\phi$ and $\psi$. 

Next, we search some locally defined conformal invariants involving the curvature tensor of ambient space.
\begin{lemma}\label{rk-con-inv}
$$K(\be,n,\be,n),~~~K(\be,\bn,\be,\bn),~~~K(\be,e,n,\bn),$$
$$K(e,n,\bn,n)+K(\be,e,e,n),~~~K(e,\bn,n,\bn)+K(\be,e,e,\bn)$$
are locally defined conformal invariants. 
\end{lemma} 
\begin{proof}
The conclusion can be verified directly by \eqref{eq-curv}. 

An alternative way to see the covariance of these quantities is that they are all equal to the values taken correspondingly by the Weyl tensor $W$, which could be seen from the following well-known formula: 
\begin{equation}\label{eq-Weyl}
{\small
\begin{split}
W(X,Y,Z,W)=&K(X,Y,Z,W)+\frac{1}{6}R\Big(h(X,Z)h(Y,W)-h(Y,Z)h(X,W)\Big)-\\
&\frac{1}{2}\Big(h(X,Z)\text{Ric}(Y,W)+h(Y,W)\text{Ric}(X,Z)-h(Y,Z)\text{Ric}(X,W)-h(X,W)\text{Ric}(Y,Z)\Big),
\end{split}
}
\end{equation}
where $R$ (resp. $\text{Ric}$) is the scalar (resp. Ricci) curvature of $(N^4,[h])$.
\end{proof}
\begin{lemma}\label{lem-con-inv}
$$2\phi_{\bar 1\bar1}+\Big(\frac{1}{4}\mathrm{Ric}(e, e)+\bH\,\bps+H\bph\Big)\phi,~~~~~~2\psi_{\bar 1\bar1}+\Big(\frac{1}{4}\mathrm{Ric}(e, e)+\bH\,\bps+H\bph\Big)\psi$$
are locally defined conformal invariants.
\end{lemma}
\begin{proof}
With respect to the conformal change $\tilde h=e^{2\lambda}h$, it is straightforward to verify that
$$\widetilde\phi_{\bar1}=e^{-2\l}(\phi_{\bar1}+\frac{\l_1+i\l_2}{2}\phi),~~~~~~\widetilde\psi_{\bar1}=e^{-2\l}(\psi_{\bar1}+\frac{\l_1+i\l_2}{2}\psi),$$
and
$$\widetilde\phi_{\bar1\bar1}=e^{-3\l}\big(\phi_{\bar1\bar1}+\frac{n(\l)\bph+\bn(\l)\bps}{2}\phi+\frac{D^2\lambda(e,e)-e(\l)e(\l)}{4}\phi\big),$$
$$\widetilde\psi_{\bar1\bar1}=e^{-3\l}\big(\psi_{\bar1\bar1}+\frac{n(\l)\bph+\bn(\l)\bps}{2}\psi+\frac{D^2\lambda(e,e)-e(\l)e(\l)}{4}\psi\big).$$
Then the conclusion follows from these identities and \eqref{eq-curv}  $\sim$  \eqref{eq-Ric-con}.
\end{proof}
\begin{remark} \label{rk-Binvariant}
Using the invariants stated in Lemma~\ref{lem-con-inv}, we can combine a globally defined conformal invariant as follows: 
\beq
\mathfrak{Re}\Big[\Big((\nabla^\perp)^2 B+\frac{1}{2}\mathrm{Ric}\otimes B \Big)(e,e,\be,\be)+4\ca(\vec H)\Big], 
\eeq
whose covariance can also be checked directly by \eqref{eq-Ric-con} and 
\begin{equation*}
(\widetilde{\nabla}^{\perp})^2 \widetilde{B}(e,e,\be,\be)=(\nabla^{\perp})^2 B(e,e,\be,\be)+B(\be,\be)(\lambda) B(e,e)+\big(D^2\l(\be,\be)-\be(\l)\be(\l)\big)B(e,e). 
\end{equation*}
\end{remark}
\begin{remark}
In the local complex coordinate $z$, the aforementioned conformal invariants 
are given by 
$$K(x_{z},n,x_{z},n),~~~ K(x_{z},\bn,x_{z},\bn),~~~K(x_{z},x_{\bar z},n,\bn),$$
$$K(x_{\bar z},n,\bn,n)+4e^{-2\s}K(x_z,x_{\bar z},x_{\bar z},n),~~~~~~K(x_{\bar z},\bn,n,\bn)+4e^{-2\s}K(x_z,x_{\bar z},x_{\bar z},\bn)$$
$$2[\Phi_{\bar z \bar z}-2(\s_{\bar z}-\bar{\rho})\Phi_{\bar z}+(\bar{\rho}^2-2\s_{\bar z}\bar{\rho}+\bar{\rho}_{\bar z})\Phi]+[K(x_{\bar z},n,x_{\bar z},\bn)+\bH\,\bPs+H\bPh]\Phi,$$
$$2[\Psi_{\bar z \bar z}-2(\s_{\bar z}+\bar{\rho})\Psi_{\bar z}+(\bar{\rho}^2+2\s_{\bar z}\bar{\rho}-\bar{\rho}_{\bar z})\Psi]+[K(x_{\bar z},n,x_{\bar z},\bn)+\bH\,\bPs+H\bPh]\Psi.$$
\end{remark}

Next, we use the invariants $\phi$ and $\psi$ to rewrite the Willmore functional $\mathcal{W}$. It follows from \eqref{eq-Hphs} that 
$$(|\phi|^2+|\psi|^2)|\theta|^2=e^{-2\s}(|\Phi|^2+|\Psi|^2)|dz|^2$$
is a globally defined conformal invariant. 
\begin{definition}
For a surface $x: M^2\to (N^4,[h])$ without umbilical points, we call
$$ g\triangleq (|\phi|^2+|\psi|^2)|\theta|^2=e^{-2\s}(|\Phi|^2+|\Psi|^2)|dz|^2$$
the M\"obius metric of $x$. 
\end{definition}
\begin{remark} \label{rk-Moebius}
It follows from \eqref{eq-tracefree} that the Willmore functional 
$$\mathcal{W}(x)=\frac 1 2\int_{M^2}  |\overset{\circ} B|^2 dA_x=2\int_{M^2} (|\phi|^2+|\psi|^2)dA_x$$
is exactly 
the volume of $M^2$ with respect to the M\"obius metric $g$ up a dilation, if $x$ is umbilic-free.  
\end{remark}
\begin{remark}
Using the Gauss equation \eqref{eq-Gauss} and Ricci equation \eqref{eq-Ricci}, we have 
$$|H|^2+K_{1212}-K_{1234}=4|\phi|^2+R_{1212}-R_{1234},~~~~~~|H|^2+K_{1212}+K_{1234}=4|\psi|^2+R_{1212}+R_{1234},$$
which implies 
the functionals 
$\mathcal{W}^+$ and $\mathcal{W}^-$ defined by Montiel and Urbano in \cite{Montiel-Urbano} can be expressed as below: 
{
$$\mathcal{W}^+(x)=4\int_{M^2} |\psi|^2dA_x+2\pi(\chi+\chi^\perp),$$
$$\mathcal{W}^-(x)=4\int_{M^2} |\phi|^2dA_x+2\pi(\chi-\chi^\perp),$$}
where $\chi$ (resp. $\chi^\perp$) is the Euler characteristic of tangent (resp. normal) bundle. {We point out that the definition of $R_{1234}$ and $K_{1234}$ here differs from the one in \cite{Montiel-Urbano} by a sign}. 
\end{remark}
\begin{remark}\label{rk-dim3-inv}
The above surface theory also applies to surfaces in 3-dimensional conformal manifold $(N^3,[h])$. In fact, we can conformally embed $(N^3,[h])$ into some $(N^4,[h])$ so that 
the imaginary part of the normal vector field $n$ is parallel, for which  
$$\bH=H,~~~\phi=\psi,~~~\mathcal{W}(x)=4\int_{M^2} |\phi|^2dA_x.$$
\end{remark}

\subsection{An Index theorem for (anti-)isotropic points} 

Let $p$ (resp. $q$) be an isolated isotropic (resp. anti-isotropic)  point for $x: M^2\to (N^4,[h])$, and  $D_p$ (resp. $D_q$) be a disc with $p$  (resp. $q$) as the only isotropic (resp. anti-isotropic) point.  We denote by $\gamma_p$ (resp. $\gamma_q$) the boundary of a disc $D_p$ (resp. $D_q$). Fix a metric $h\in[h]$, define the index of $p$  and $q$ by
\beq\label{eq-index}
 \mathrm{Ind}_p\phi=\frac{1}{2\pi i}\int_{\gamma_p}d\log \phi,~~~~~~\mathrm{Ind}_q\psi=\frac{1}{2\pi i}\int_{\gamma_q}d\log \psi.
 \eeq
It is clear 
from \eqref{eq-Hphs} that $\mathrm{Ind}_p\phi$ and $\mathrm{Ind}_q\psi$ are conformal invariants. 
\par\medskip
It is easy to check that
$$\Theta_1\triangleq d(\log(\phi))-i(2\q_{12}-\q_{34}),~~~~~~~~~~~~~\Theta_2\triangleq d(\log(\psi))-i(2\q_{12}+\q_{34})$$
are two globally defined 1-froms on $M^2\backslash\{\phi=0\}$ and $M^2\backslash\{\psi=0\}$, respectively.
\begin{theorem}\label{thm-index}
Let $x:M^2\rightarrow (N^4,[h])$ be a closed oriented surfaces. If its isotropic points are isolated, which are denoted by $\{p_j\}$, then
$$\sum_{j} \mathrm{Ind}_{p_j}\phi=-2\chi+\chi^\perp.$$
If its anti-isotropic points are isolated, which are denoted by $\{q_k\}$, then
$$\sum_{k} \mathrm{Ind}_{q_k}\psi=-2\chi-\chi^\perp.$$
\end{theorem}
\begin{proof}
We just give the proof of the first index formula. Let $D_j$ be a small disc containing $p_j$ as the only isotropic point such that $\{D_j\}$ are disjoint.  The boundary of $D_j$ is denoted by $\gamma_j$. Set $M^*=M^2\backslash\{\cup_jD_j\}$, then there is no isotropic point in $M^*$ and $\partial M^*=\cup_j\gamma_j^-$, where we use the superscript $-$ to indicate that the orientation
on $\gamma_j$ induced from $M^*$ and $D_j$ are different. It follows from the definition that
\begin{equation*}
d\Theta_1=-id(2\q_{12}-\q_{34})=i(2R_{1212}-R_{1234})dA_x.
\end{equation*}
By Stokes formula we have 
\begin{equation}\label{eq-index-int}
i\int_{M^*}(2R_{1212}-R_{1234})dA_x=\int_{M^*}d\Theta_1=\int_{\partial M^*}\Theta_1=-\sum_j \left(\int_{\gamma_j}d\log\Phi -i\int_{\gamma_j}(2\q_{12}-\q_{34})\right), 
\end{equation}
\begin{equation}\label{eq-index-dTheta1}
\int_{\gamma_j}(2\q_{12}-\q_{34})=\int_{D_j}d(2\q_{12}-\q_{34})=-\int_{D_j}(2R_{1212}-R_{1234})dA_x.
\end{equation}
Combining \eqref{eq-index}, \eqref{eq-index-int} and \eqref{eq-index-dTheta1}, we can obtain 
$$ 2i\int_{M}R_{1212}dV_h-i\int_{M}R_{1234}dV_h=-2\pi\,i \sum_j\mathrm{Ind}_{p_j}\Phi. $$
The conclusion now follows from the Gauss-Bonnet formula.
\end{proof}
Applying this theorem to the $3$-dimensional case, we can obtain the following well-known result.  
\begin{corollary}
Let $x:M^2\rightarrow (N^3,[h])$ be a closed oriented surfaces. If its umbilical points are isolated, which are denoted by $\{p_j\}$, then
$$\sum_{j} \mathrm{Ind}_{p_j}\phi=-2\chi.$$
\end{corollary}

\section{The first variation of the Willmore functional}\label{1st-variation}
In this section, we calculate the first variation of the Willmore functional for closed surfaces in $4$-dimensional conformal manifold. The computation also applies to the variation satisfying suitable boundary conditions for compact surfaces with boundary.  In some special  $4$-dimensional Riemannian manifolds, whether minimal surfaces are Willmore is  discussed. 

Let $x: M^2\to (N^4,[h])$ be a closed surface. Consider a variation of $x$
$$X: M^2\times {\mathbb R}\to (N^4,[h])$$
such that $x_t=X(\cdot,t): M^2\to (N^4,[h])$ is a surface for each $t$ and $x_0=x$. The variation vector field
$V$ is defined by $V=dx(\frac{\partial}{\partial t})$. We point out that the following calculation also applies to compact surfaces with boundary if we further require the variation to satisfy $x_t=x$ and $dx_t=dx$ on the boundary.  

Fix $h\in[h]$, as in Section~\ref{sec-con-inv}, we choose a local orthonormal frame $\{e_1,e_2, e_3, e_4\}$ for $(N^4,h)$ with dual basis $\{\o_1,\o_2, \o_3, \o_4\}$, such that on $x_t(M^2)$, $\{e_1, e_2\}$ (resp. $\{e_3, e_4\}$) are tangent (resp. normal) vector fields. The complex notation defined in \eqref{eq-complex} will be used.

We write
\begin{align}
&X^*\o=x^*_t\o+\epsilon\,dt=\q+\epsilon\,dt, \label{eq-tangentvari}\\
&X^*\O=x_t^*\O+v\,dt=v\,dt,\\
&X^*\o_{12}=x_t^*\o_{12}+B_1\,dt=\theta_{12}+B_1\,dt,\label{eq-xo12}\\
&X^*\o_{34}=x_t^*\o_{34}+B_2\,dt=\theta_{34}+B_2\,dt,\label{eq-xo34}\\
&X^*(\II_1)=x_t^*\II_{1}+A_1\,dt=H\q+2\bps\,\bq+A_1\,dt, \label{eq-xii1}\\
&X^*(\II_2)=x_t^*\II_{2}+A_2\,dt=H\bq+2\phi\,\q+A_2\,dt, \label{eq-xii2}
\end{align}
where $\e, v, A_1, A_2$ (resp. $B_1, B_2$) are locally defined complex-valued (resp. real-valued) functions on $M^2$.
On $T^*(M^2\times \mathbb{R})$, we decompose the exterior differentiation $d$  as
$$d=d_m+dt\wedge\frac{\partial}{\partial t}.$$
It follows from \eqref{eq-do}  $\sim$  \eqref{eq-dII2} that
\begin{equation}\label{eq-pqt}
\frac{\partial \theta}{\partial t}=d_m \e+i\e\theta_{12}-(iB_1+\frac{v\bH}{2}+\frac{\bv H}{2})\q-(v\bph+\bv\bps)\bq.
\end{equation}
\begin{equation}\label{eq-dv}
d_m\,v+i\q_{34}v=\frac{1}{2}(A_2-\bep H-2\,\e\,\phi)\q+\frac{1}{2}(A_1-\e H-2\,\bep\,\bps)\bq,
\end{equation}
\begin{equation}\label{eq-pII1t}
\begin{split}
\frac{\partial (x_t^*\II_1)}{\partial t}=d_m A_1+i(\theta_{12}+\theta_{34})A_1-i(B_1+B_2)x_t^*\II_1+\frac{1}{4}\Big[K(e,n,\be,e)(\bep\q-\e\bq)\\
+K(e,n,\be,n)\bv\q+K(e,n,\be,\bn)v\q+K(e,n,e,n)\bv\bq+K(e,n,e,\bn)v\bq\Big],
\end{split}
\end{equation}
\begin{equation}\label{eq-pII2t}
\begin{split}
\frac{\partial (x_t^*\II_2)}{\partial t}=d_m A_2-i(\theta_{12}-\theta_{34})A_2+i(B_1-B_2)x_t^*\II_2+\frac{1}{4}\Big[K(\be,n,\be,e)(\bep\q-\e\bq)\\
+K(\be,n,\be,n)\bv\q+K(\be,n,\be,\bn)v\q+K(\be,n,e,n)\bv\bq+K(\be,n,e,\bn)v\bq\Big].
\end{split}
\end{equation}
Using \eqref{eq-dv}, we have 
\beq\label{eq-v1A1}
v_{\bar1}=\frac{1}{2}(A_1-\e H-2\,\bep\,\bps),~~~~~~v_1=\frac{1}{2}(A_2-\bep H-2\,\e\,\phi),
\eeq
where $v_{\bar1}$ and $v_1$ are the covariant derivatives of $v$ as defined in \eqref{eq-cov}.  

It is easy to see that
$$\mathcal{W}(x_t)=2\int_{M^2} (|\phi|^2+|\psi|^2)dA_{x_t}=\frac{1}{4i}\int_{M^2}(x_t^*\II_1\wedge x_t^*\overline{\II_1}-x_t^*\II_2\wedge x_t^*\overline{\II_2}+2|H|^2\bq\wedge\q).$$
Note that $\bq\wedge H\q=\bq\wedge x_t^*\II_1=x_t^*\II_2\wedge \q$, using which we have  
\begin{equation}\label{eq-pHq}
\begin{split}
\frac{\partial}{\partial t}(|H|^2\bq\wedge\q)&=\bH \bq\wedge \frac{\partial (H\q)}{\partial t}-H\q\wedge \frac{\partial (\bH \bq)}{\partial t}\\
&=2i\mathfrak{Im}\Big(\bH\bq\wedge\frac{\partial (x_t^*\II_1)}{\partial t}-\bH(x_t^*\II_1-H\q)\wedge\frac{\partial \bq}{\partial t}\Big)\\
&=2i\mathfrak{Im}\Big(\frac{\partial (x_t^*\II_2)}{\partial t}\wedge\bH\q+\bH(x_t^*\II_2-H\bq)\wedge\frac{\partial \q}{\partial t}\Big),
\end{split}
\end{equation}
where $\mathfrak{Im}$ means taking the imaginary part. 
From 
\eqref{eq-II} 
and \eqref{eq-pHq}, we can obtain 
\begin{equation*}
\begin{split}
&\frac{\partial}{\partial t}(x_t^*\II_1\wedge x_t^*\overline{\II_1}-x_t^*\II_2\wedge x_t^*\overline{\II_2}+2|H|^2\bq\wedge\q)\\
=&2\mathfrak{Im}\Big[\frac{\partial (x_t^*\II_1)}{\partial t}\wedge(x_t^*\overline{\II_1}-\bH\,\bq)-\bH(x_t^*\II_1-H\q)\wedge\frac{\partial \bq}{\partial t}\Big]-\\
&2\mathfrak{Im}\Big[\frac{\partial (x_t^*\II_2)}{\partial t}\wedge(x_t^*\overline{\II_2}-\bH\,\q)-\bH(x_t^*\II_2-H\bq)\wedge\frac{\partial \q}{\partial t}\Big]\\
=&4\mathfrak{Im}\Big(\psi\,\frac{\partial (x_t^*\II_1)}{\partial t}\wedge\theta-\bH\,\bps\,\bq\wedge\frac{\partial \bq}{\partial t}-\bph\,\frac{\partial (x_t^*\II_2)}{\partial t}\wedge\bq+\bH\,\phi\,\q\wedge\frac{\partial \q}{\partial t}\Big).
\end{split}
\end{equation*}
Then it follows from \eqref{eq-pqt}  $\sim$  
\eqref{eq-v1A1} that
\begin{equation*}
\begin{split}
&\frac{\partial}{\partial t}(x_t^*\II_1\wedge x_t^*\overline{\II_1}-x_t^*\II_2\wedge x_t^*\overline{\II_2}+2|H|^2\bq\wedge\q)\\
=&8\mathfrak{Im}\Big(\psi\,\big(d_m v_{\bar1}+i(\theta_{12}+\theta_{34})v_{\bar1}\big)\wedge\q+d_m(\bep|\psi|^2\q)-\\
&~~~~~~~\bph\,\big(d_m v_{1}-i(\theta_{12}-\theta_{34})v_{1}\big)\wedge\bq-\,d_m(\epsilon|\phi|^2\bq)\Big)+\\
&4\mathfrak{Re}\Big(\frac{v\bps}{4}K(\be,\bn,\be,\bn)+v\,\psi\big(\frac{1}{4}\mathrm{Ric}(e, e)+\bH\,\bps+H\bph\big)+\\
&~~~~~~~\frac{v\phi}{4}K(e,\bn,e,\bn)+v\,\bph\big(\frac{1}{4}\mathrm{Ric}(\be, \be)+H\,\psi+\bH\phi\big)\Big)\bq\wedge\q,
\end{split}
\end{equation*}
where we have used 
$$K(e,n,e,\bn)=\mathrm{Ric}(e, e), ~~~K(\be,\bn,\be,n)=\mathrm{Ric}(\be, \be).$$
By direct calculation, it is easy to verify that
$$\psi\,(d_m v_{\bar1}+i(\theta_{12}+\theta_{34})v_{\bar1})\wedge\q=v\,\psi_{\bar1\bar1}\,\bq\wedge\q+d_m(v_{\bar1}\psi\q)-d_m(v\psi_{\bar1}\q),$$
$$\bph\,\big(d_m v_{1}-i(\theta_{12}-\theta_{34})v_{1}\big)\wedge\bq=-v\,\bph_{11}\,\bq\wedge\q+d_m(v_{1}\bph\,\bq)-d_m(v\bph_{1}\bq),$$
where $\bph_1=\overline{\phi_{\bar1}},~\bph_{11}=\overline{\phi_{\bar1\bar1}}$. Then we complete the calculation of the first variation. 
\begin{proposition}
The first variation formula of the Willmore functional is
\begin{equation}\label{eq-1stvar}
\begin{split}
\frac{\p}{\p t}\mathcal{W}(x_t)=2\mathfrak{Re}\Big\{\int_{M^2}v\Big[&2\psi_{\bar1\bar1}+\big(\frac{1}{4}\mathrm{Ric}(e, e)+\bH\,\bps+H\bph\big)\psi+\frac{1}{4}K(\be,\bn,\be,\bn)\bps+\\
&2\bph_{11}+\big(\frac{1}{4}\mathrm{Ric}(\be, \be)+H\,\psi+\bH\phi\big)\bph+\frac{1}{4}K(e,\bn,e,\bn)\phi\Big]dA_{x_t}\Big\},
\end{split}
\end{equation}
which implies the Euler-Lagrange equation is
\begin{equation}\label{eq-Eul}
\begin{split}
&\Big(2\psi_{\bar1\bar1}+\big(\frac{1}{4}\mathrm{Ric}(e, e)+\bH\,\bps+H\bph\big)\psi\Big)+\frac{1}{4}K(\be,\bn,\be,\bn)\bps\\
+&\Big(2\bph_{11}+\big(\frac{1}{4}\mathrm{Ric}(\be, \be)+H\,\psi+\bH\phi\big)\bph\Big)+\frac{1}{4}K(e,\bn,e,\bn)\phi=0.
\end{split}
\end{equation}
\end{proposition}
\begin{remark}\label{rk-conEuL}
It follows from Lemma~\ref{rk-con-inv} and Lemma~\ref{lem-con-inv} that the four terms in \eqref{eq-Eul} are locally defined conformal invariants. 

Using Remark~\ref{rk-Binvariant}, we can also write down the Euler-Lagrange equation of 
$\mathcal{W}$ by globally defined conformal invariants as follows: 
\beq\label{eq-Glcon}
\mathfrak{Re}\Big[\left((\nabla^\perp)^2 B+\frac{1}{2}Ric\otimes B\right)(e,e,\be,\be)+4\ca(\vec H)+W(e,B(\be,\be),e)^\perp\Big]=0, 
\eeq
where $W(\cdot,\cdot,\cdot)$ denotes 
the Weyl tensor of $(1,3)$ type.

In the local complex coordinate $z$, \eqref{eq-Eul} can be rewritten as
\begin{equation*}\label{eq-Eul-com}
\begin{split}
\!\!\!\!\!&\Big[2\big(\Psi_{\bar z \bar z}-2(\s_{\bar z}+\bar{\rho})\Psi_{\bar z}+(\bar{\rho}^2+2\s_{\bar z}\bar{\rho}-\bar{\rho}_{\bar z})\Psi\big)+\big(\mathrm{Ric}(x_{\bar z},x_{\bar z})+\bH\,\bPs+H\bPh\big)\Psi\Big]+K(x_{z},\bn,x_{z},\bn)\bPs\\
\!\!\!\!\!+&\Big[2\big(\bPh_{z z}-2(\s_{z}-\rho)\bPh_z+(\rho^2-2\s_{z}\rho+\rho_{z})\bPh\big)+\big(\mathrm{Ric}(x_{ z},x_{ z})+H\,\Psi+\bH\Phi\big)\bPh\Big]+K(x_{\bar z},\bn,x_{\bar z},\bn)\Phi=0.
\end{split}
\end{equation*}
\end{remark}

In order to discuss whether minimal surfaces are Willmore, we also write down the Euler-Lagrange equation of $\mathcal{W}$ as follows: 
\begin{equation}\label{eq-Eul-H}
\begin{split}
&\bH_{1\bar1}+\bH_{\bar1 1}+\frac{\bH}{4}\Big(6K_{1212}-\text{Ric}(e,\be)+4\big(|\psi|^2+|\phi|^2\big)\Big)-\frac{H}{4}\Big(\text{Ric}(\bn,\bn)-8\bph\psi\Big)\\
+&\frac{1}{8}(\nabla_{\bn} K)(e,\be,e,\be)+\frac{1}{2}K(\be,\bn,\be,\bn)\bps+\frac{1}{2}K(e,\bn,e,\bn)\phi+\frac{1}{2}\text{Ric}(e,e)\psi+\frac{1}{2}\text{Ric}(\be,\be)\bph=0,
\end{split}
\end{equation}
which just follows from \eqref{eq-sderi}, the second Bianchi identity, $K(\be,n,e,n)=-\text{Ric}(n,n),$ and
$$K(\bn,n,\be,e)+K(\be,n,\bn,e)-K(\be,e,\be,e)=6K_{1212}-K_{1234}-\text{Ric}(e,\be),$$
$$K(n,\bn,\be,e)+K(\be,\bn,n,e)-K(\be,e,\be,e)=6K_{1212}+K_{1234}-\text{Ric}(e,\be).$$
Using \eqref{eq-lapH}, we can also write \eqref{eq-Eul-H} as the following global form:  
$$\triangle^\perp \vec H+\ca(\vec H)+\frac{1}{2}\mathfrak{Re}\Big(\nabla_eK(\be,e,\be)+2K(\be, B(e,e),\be)-2K(\be, \vec H, e)\Big)^\perp-\frac{1}{2}K(e,\be,e,\be)\vec H=0.$$

For a minimal surface in $4$-dimensional locally symmetric space $(N^4, h)$, to be Willmore, it should satisfy
$$\mathfrak{Re}\big(K(\be, B(e,e),\be)\big)^\perp=0.$$
It follows from \eqref{eq-Weyl} that if, in addition, $(N^4, h)$ is also Einstein, 
then the Willmore equation of minimal surfaces becomes 
\beq \label{eq-Einstein}
\mathfrak{Re}\big(W(\be, B(e,e),\be)\big)^\perp=0.
\eeq
Known examples of such kind of spaces contain  

(1) $4$-manifolds locally isometric to a product of two surfaces with equal constant curvature; 

(2) Locally irreducible, locally symmetric $4$-manifolds; 

(3) $4$-dimensional  
compact Einstein manifolds with nonnegative curvature operator, or nonnegative isotropic curvature (it was proved such manifolds are  locally symmetric, see \cite{Tachibana} and \cite{Micallef-Wang}). 


As stated before, Montiel and Urbano proved in \cite{Montiel-Urbano} that super-minimal surfaces in $\mathbb{C}P^2$ with positive spin are Willmore. The following proposition can be seen a generalization of this result. 
\begin{proposition}\label{prop-einstein}
In locally symmetric, self-dual (resp. anti-self-dual) Einstein $4$-manifolds, 
super-minimal surfaces with positive spin (resp. negative spin) are Willmore. 
\end{proposition}
\begin{proof}
We use the notations given in Remark~\ref{rk-orientaion}. 
Locally, \eqref{eq-Einstein} is equivalent to 
$$W(\be,\bn,\be,\bn)\bps+W(e,\bn,e,\bn)\bph=0,$$
where we have used $W(e,\bn,e,n)=0$ and $W(\be,\bn,\be,n)=0$, which follows from the trace-free property of the Weyl tensor. 

We only prove the self-dual case. For super-minimal surfaces with positive spin, we have $\psi=0$.  Note that $e\wedge \bn \in \Lambda^2_{-} (TM^2)$ (the eigenspace of the Hodge star operator $*$ with eigenvalue $1$). 
Therefore, we can derive that $W(e,\bn,e,\bn)=W^{-}(e,\bn,e,\bn)=0$, which completes the proof. 
\end{proof}
Due to the work of Derdzi{\'n}ski \cite{Derdzinski}, $4$-manifolds stated as follows satisfy the assumption of Proposition~\ref{prop-einstein}.    

(1) Orientable, locally irreducible, and locally symmetric $4$-manifolds (such manifolds are proved to be self-dual for some orientaion); 

(2) Self-dual K\"ahler-Einstein $4$-manifolds. 

Concrete examples of self-dual K\"ahler-Einstein $4$-manifolds contain complex space forms, and Ricci-flat k\"ahler surfaces (such as the $K_3$-surface endowed with any K\"ahler metric provided by positive solution to the Calabi conjecture). Since complex curves and totally real minimal surfaces (minimal Lagrangian surfaces) in self-dual K\"ahler-Einstein $4$-manifolds are super-minimal surfaces with positive spin, it follows from Proposition~\ref{prop-einstein} that they are all Willmore.  
\begin{remark}
For surfaces in 3-dim conformal manifold $(N^3, [h])$, it follows from Remark~\ref{rk-dim3-inv} and \eqref{eq-Eul-com} that the Euler-Lagrange equation of Willmore functional is
\begin{equation}
\mathfrak{Re}\Big[\big(\Phi_{\bar z \bar z}-2\s_{\bar z}\Phi_{\bar z}\big)+\big(K(x_{\bar z},n,x_{\bar z},n)+H\bPh\big)\Phi\Big]=0.
\end{equation}
\end{remark}

\section{Willmore $2$-spheres in $4$-dimensional locally symmetric spaces}\label{sec-2sphere}
In this section, we continue to study whether minimal surfaces in $4$-dimensional locally symmetric spaces are Willmore. Some holomorphic differentials are constructed firstly, with which we proved that to be Willmore a  minimal $2$-sphere is required to be super-minimal.  

Let $z$ be a local complex coordinate on $(M^2, x*h)$, the notation given at the end of subsection~\ref{subsec-cov} will be used. Firstly, we restate \eqref{eq-stru} as follows: 
\begin{equation}\label{eq-stru-complex}
\begin{split}
&\nabla_{\xz} \,\xz=2\s_z\,\xz+\frac{1}{2}\Psi\,n+\frac{1}{2}\Phi\,\bn,~~~~~~~~~~~~\nabla_{\xz}\, \xbz=\frac{e^{2\s}}{4}\bH\,n+\frac{e^{2\s}}{4}H\,\bn,\\
&\nabla_{\xz} \,n=-H\, \xz-2e^{-2\s}\Phi\,\xbz+\r\, n,~~~~~~~~~\nabla_{\xbz}\, n=-H\,\xbz-2e^{-2\s}\bPs\,\xz-\br\, n.
\end{split}
\end{equation}
By direct calculation, we can obtain
\beq\label{eq-1holo}
\frac{\partial}{\partial {\bar z}}\Big(K(x_{z},n,x_{z},\bn)+(\bH\Phi+H\Psi)\Big)\!\!=\!\!(\nabla_{x_{\bar z}}K)(x_{z},n,x_{z},\bn)-\frac{e^{2\s}}{2}K(x_{ z},\bn,\vec{H},n+\bn)+(\bH_{\bar z}-\overline{\r} \bH)\,\Phi+(H_{\bar z}+\overline\r H)\Psi.
\eeq
\beq\label{eq-2holo}
\frac{\partial}{\partial {\bar z}}K(x_{z},\bn,x_{z},\bn)=2\br K(x_{z},\bn,x_{z},\bn)+(\nabla_{x_{\bar z}}K)(x_{z},\bn,x_{z},\bn)+\frac{e^{2\s}\bH}{2}K(n,\bn,x_{ z},\bn)-2\bH K(x_{z},\xbz,\xz,\bn).
\eeq
\beq\label{eq-3holo}
\frac{\partial}{\partial {\bar z}}K(x_{z},n,x_{z},n)=-2\br K(x_{z},n,x_{z},n)+(\nabla_{x_{\bar z}}K)(x_{z},n,x_{z},n)+\frac{e^{2\s}H}{2}K(\bn,n,x_{z},n)-2H K(\xz,\xbz,\xz,n).
\eeq
\beq\label{eq-4holo}
\begin{split}
\frac{\partial}{\partial {\bar z}}K(x_{z},\bn,\xbz,\xz)=&(2\s_z+\br)K(x_{z},\bn,\xbz,\xz)+(\nabla_{x_{\bar z}}K)(x_{z},\bn,\xbz,\xz)\\
&-\frac{\bPh}{2}K(\xz,\bn,\xz,n)-\frac{\bPs}{2}K(\xz,\bn,\xz,\bn)+\frac{e^{2\s}H}{4}K(\xz,\bn,\xbz,\bn)\\
&+\bH \Big(\frac{e^{2\s}}{4}K(\xbz,\xz,n,\bn)+K(\xz,\xbz,\xz,\xbz)+\frac{e^{2\s}}{4}K(\xz,\bn,\xbz,n)\Big).
\end{split}
\eeq
\beq\label{eq-5holo}
\begin{split}
\frac{\partial}{\partial {\bar z}}K(x_{z},n,\xz,\xbz)=&(2\s_z-\br)K(x_{z},n,\xz,\xbz)+(\nabla_{x_{\bar z}}K)(x_{z},n,\xz,\xbz)\\
&+\frac{\bPs}{2}K(\xz,\bn,\xz,n)+\frac{\bPh}{2}K(\xz,n,\xz,n)-\frac{e^{2\s}\bH}{4}K(\xz,n,\xbz,n)\\
&+H \Big(\frac{e^{2\s}}{4}K(\xbz,\xz,n,\bn)-K(\xz,\xbz,\xz,\xbz)-\frac{e^{2\s}}{4}K(\xz,n,\xbz,\bn)\Big).
\end{split}
\eeq
\beq\label{eq-6holo}
\begin{split}
\frac{\partial}{\partial { z}}K(x_{z},n,\xz,n)=&2(2\s_z+\r)K(x_{z},n,\xz,n)+(\nabla_{x_{z}}K)(x_{z},n,\xz,n)~~~~~~~~~~~~~~~~~~~~~~\\
&-\Big(K(\xz,n,n,\bn)+4e^{-2\s}K(\xz,\xbz,\xz,n)\Big)\Phi,
\end{split}
\eeq
\beq\label{eq-7holo}
\begin{split}
\frac{\partial}{\partial { z}}K(x_{z},\bn,\xz,\bn)=&2(2\s_z-\r)K(x_{z},\bn,\xz,\bn)
+(\nabla_{x_{z}}K)(x_{z},\bn,\xz,\bn)~~~~~~~~~~~~~~~~~~~~~~\\
&+\Big(K(\xz,\bn,n,\bn)-4e^{-2\s}K(\xz,\xbz,\xz,\bn)\Big)\Psi.
\end{split}
\eeq
\beq\label{eq-8holo}
\begin{split}
\frac{\partial}{\partial {z}}K(\xz,n,x_{z},\bn)=&(\nabla_{x_{z}}K)(x_{z},\bn,\xz,n)+\frac{\Psi}{2}\Big(K(\xz,n,n,\bn)-4e^{-2\s}K(\xz,\xbz,\xz,n)\Big)\\
&+4\s_z K(x_{z},\bn,\xz,n)
-\frac{\Phi}{2}\Big(K(\xz,\bn,n,\bn)+4e^{-2\s}K(\xz,\xbz,\xz,\bn)\Big).
\end{split}
\eeq
\begin{lemma}\label{lem-1diff}
Let $x: M^2\rightarrow (N^4,h)$ be a surface in locally symmetric space, if its mean curvature vector field is parallel, then the globally defined differential
$$\Big(K(x_{z},n,x_{z},\bn)+(\bH\Phi+H\Psi)\Big)dz^2$$
is holomorphic.
\end{lemma}
\begin{proof}
If the mean curvature vector field $\vec{H}=(H\bn+\bH n)/2$ is parallel, then
$$\bH_{\bar z}-\overline{\r} \bH=0,~~~~~~H_{\bar z}+\overline\r H=0,~~~~~~K(x_{ z},\bn,\vec{H},n+\bn)=0.$$
The conclusion now follows from \eqref{eq-1holo}.
\end{proof}
\begin{lemma}\label{lem-2diff}
Let $x: M^2\rightarrow (N^4,h)$ be a surface in locally symmetric space, if it is minimal, then the following globally defined differentials
$$K(x_{z},n,x_{z},\bn)dz^2,~~~~~~\Big(K(x_{z},\bn,x_{z},\bn)K(x_{z},n,x_{z},n)\Big)dz^4$$
are holomorphic. Moreover, $K(x_{z},\bn,x_{z},\bn)$ either vanishes identically or has only isolated zero point, so does  $K(x_{z},n,x_{z},n)$.
\end{lemma}
\begin{proof}
The holomorphicity of the first differential follows from Lemma~\ref{lem-1diff}. 

Using \eqref{eq-2holo} and \eqref{eq-3holo}, we can obtain
$$\frac{\partial}{\partial {\bar z}}\Big(K(x_{z},\bn,x_{z},\bn)K(x_{z},n,x_{z},n)\Big)=0,$$
under the assumption that $H=0$. 

The last conclusion follows from Chern's famous theorem on elliptic differential systems given in \cite{Chern}.
\end{proof}
\begin{remark}
It follows from \eqref{eq-Eul-H} that, minimal surfaces in locally symmetric space is Willmore if and only if
\begin{equation}\label{eq-mini-Will}
K(x_{\bar z},\bn, x_{\bar z},n)\Psi+
K(x_{z},n,x_{z},\bn)\bPh+K(x_{z},\bn,x_{z},\bn)\bPs+K(x_{\bar z},\bn,x_{\bar z},\bn)\Phi=0.
\end{equation}
\end{remark}
Next, we apply these holomorphic differentials to minimal $2$-spheres in locally symmetric spaces of dimension $4$. 
\begin{proposition}
Let $x: \mathbb{S}^2\rightarrow (N^4,h)$ be a minimal $2$-sphere in locally symmetric space. If $(N^4, h)$ is furthrt conformally flat, then $x$ is Willmore.
\end{proposition}
\begin{proof}
Since holomorphic differentials on the Riemannian sphere always vanishes. It follows from Lemma~\ref{lem-2diff} that $K(x_{z},n,x_{z},\bn)=0$. 

On the other hand, if $(N^4,h)$ is conformally flat, then its Weyl tensor vanishes.  The conclusion follows from  
$$K(x_{z},\bn,x_{z},\bn)=W(x_{z},\bn,x_{z},\bn)=0,~~~~~~K(x_{\bar z},\bn,x_{\bar z},\bn)=W(x_{\bar z},\bn,x_{\bar z},\bn)=0,$$
where \eqref{eq-Weyl} has been used. 
\end{proof}
Note that $3$-dim locally symmetric space is  automatically conformally flat. Therefore, we can obtain the following corollary. 
\begin{corollary}
Minimal $2$-spheres in $3$-dimensional locally symmetric spaces are Willmore.
\end{corollary}
\begin{theorem}
Let $x: \mathbb{S}^2\rightarrow (N^4,h)$ be a minimal sphere in locally symmetric space. If it is Willmore, then it must be super-minimal.
\end{theorem}
\begin{proof}
It follows from Lemma~\ref{lem-2diff} that
$$K(x_{z},n,x_{z},\bn)=0,~~~~~~K(x_{z},\bn,x_{z},\bn)K(x_{z},n,x_{z},n)=0,$$
and then either $K(x_{z},\bn,x_{z},\bn)$ or $K(x_{z},n,x_{z},n)$ vanishes identically.

Firstly, we assume they don't vanish identically at the same time, for example, $K(x_{z},\bn,x_{z},\bn)=0$ everywhere but $K(x_{z},n,x_{z},n)$ has non-zero points. If $x$ is Willmore, then from \eqref{eq-mini-Will}, we have
$$K(x_{\bar z},\bn,x_{\bar z},\bn)\Phi=0.$$
So whenever $K(x_{z},n,x_{z},n)\neq 0$, we have $\Phi=0$. It follows from Lemma~\ref{lem-2diff} that the non-zero points of $K(x_{z},n,x_{z},n)$ constitute an open dense subset of $\mathbb{S}^2$, on which we derive that $\Phi=0$. Therefore $\Phi$ vanishes identically on $\mathbb{S}^2$, which implies $x$ is super-minimal.

If both $K(x_{z},\bn,x_{z},\bn)$ and $K(x_{z},n,x_{z},n)$ vanish identically. Then using Chern's famous theorem on elliptic differential system again, from \eqref{eq-Coda-complex}, \eqref{eq-4holo} and \eqref{eq-5holo}, we can obtain that 
$$\Phi,~~~\Psi,~~~K(x_{z},\bn,\xbz,\xz),~~~K(x_{z},n,\xz,\xbz)$$
either vanish identically or equal zero only at isolated points. 
Set 
$$\Omega\triangleq \{p\in \mathbb{S}^2\,\big|\, \Phi(p)\neq0, \Psi(p)\neq0\}.$$
Assume $x$ is not super-minimal, then $\Omega$ is open and dense in $\mathbb{S}^2$. It follows from \eqref{eq-6holo}  $\sim$  \eqref{eq-8holo} that
$$\Big(K(\xz,n,n,\bn)+4e^{-2\s}K(\xz,\xbz,\xz,n)\Big)\Phi=0, 
$$
$$\Big(K(\xz,\bn,n,\bn)-4e^{-2\s}K(\xz,\xbz,\xz,\bn)\Big)\Psi=0, 
$$
$$\frac{\Psi}{2}\Big(K(\xz,n,n,\bn)-4e^{-2\s}K(\xz,\xbz,\xz,n)\Big)
-\frac{\Phi}{2}\Big(K(\xz,\bn,n,\bn)+4e^{-2\s}K(\xz,\xbz,\xz,\bn)\Big)=0.
$$
So on $\Omega$ and hence then $\mathbb{S}^2$, we have
\beq\label{eq-psKphK}
\Psi K(\xz,\xbz,\xz,n)
-\Phi K(\xz,\xbz,\xz,\bn)=0.
\eeq
On the other hand, it follows from \eqref{eq-4holo} and \eqref{eq-5holo} that
$$\frac{\partial}{\partial {\bar z}}\Big(e^{-4\s}K(x_{z},\bn,\xbz,\xz)K(x_{z},n,\xz,\xbz)\Big)=0,$$
which implies $e^{-4\s}K(x_{z},\bn,\xbz,\xz)K(x_{z},n,\xz,\xbz)dz^2$ is a globally defined holomorphic differential.  On the minimal $2$-sphere $x(\mathbb{S}^2)$, this differential vanishes. Combining with \eqref{eq-psKphK}, we obtain that both $K(x_{z},\bn,\xbz,\xz)$ and $K(x_{z},n,\xz,\xbz)$ vanish on $\Omega$ and hence then the whole sphere $\mathbb{S}^2$.
Now it follows from the Codazzi equation \eqref{eq-Coda-complex} that
$$\frac{\partial}{\partial { \bar z}}(\Phi\Psi)=0,$$
which implies $\Phi\Psi dz^4$ is a globally defined holomorphic differential on $\mathbb{S}^2$, so it must vanish identically. This implies $\Omega$ is empty, which is contradicted with our assumption.

\end{proof}

\section{The second variation of the Willmore functional}\label{2nd-variation}
In this section, we calculate the second variation of the Willmore functional for closed surfaces in $4$-dimensional conformal manifold. 

Let $x: M^2\to (N^4,[h])$ be a closed surface. We continue to use the notations given in Section~\ref{sec-con-inv} and Section~\ref{1st-variation}. It follows from Section~\ref{1st-variation} that the Willmore functional only depends on the normal variation. Therefore in the following, we assume the function $\epsilon$ appearing in \eqref{eq-tangentvari} equals zero identically, i.e., the variation vector field $V=dx(\frac{\p}{\p t})$ is given by
$$V=\frac{\bv n+v\bn}{2}.$$

From the structure equation of $(N^4, h)$ and \eqref{eq-pII1t}  $\sim$  
\eqref{eq-v1A1}, we can obtain the variation of $e$ and $n$ as follows:  
\beq\begin{split}
&\frac{\p e}{\p t}=\nabla_V e=
\vb1\,\bn+\bv_{\bar1}\,n-iB_1e,\\
&\frac{\p n}{\p t}=\nabla_V n=
-\vb1\,\be-v_1\,e-iB_2n,
\end{split}
\eeq
where $v_1$ and $v_{\bar 1}$ are the covariant derivative of $v$ as defined in \eqref{eq-cov}. 
By direct calculation, we have
\begin{equation}\label{eq-vKenebn}
\begin{split}
\frac{\p}{\p t} \mathrm{Ric}(e, e)=\frac{\p}{\p t} K(e, n, e,\bn)=&(\nabla_V K)(e,n,e,\bn)
+\vb1K(e,\bn,\bn,n)-\vb1K(e,\be,e,\bn)\\
&+\bv_{\bar1}K(e,n,n,\bn)-\bv_{\bar1}K(e,\be,e,n)-2i\,B_1\,\mathrm{Ric}(e, e),
\end{split}
\end{equation}
\beq
\begin{split}
\frac{\p}{\p t} K(\be,\bn,\be,\bn)=(\nabla_VK)(\be,\bn,\be,\bn)
+2\bv_1K(\be,\bn,n,\bn)-2\bv_1K(\be,e,\be,\bn)+2i(B_1+B_2)K(\be,\bn,\be,\bn), 
\end{split}
\eeq
\beq
\begin{split}
\frac{\p}{\p t} K(e,\bn,e,\bn)=(\nabla_VK)(e,\bn,e,\bn)
+2\bv_{\bar1}K(e,\bn,n,\bn)+2\bv_{\bar1}K(\be,e,e,\bn)-2i(B_1-B_2)K(e,\bn,e,\bn).
\end{split}
\eeq
We also need the following variation formulas of $\q_{12},\q_{34}, H, \phi, \psi$, which can be derived from the differential of \eqref{eq-xo12}, \eqref{eq-xo34} and  \eqref{eq-pqt}  $\sim$  \eqref{eq-pII2t}.
\vskip -0.3cm
\beq\label{eq-vthe12}
\begin{split}
\frac{\p \q_{12}}{\p t}=&d_{m}B_1+\frac{1}{2}\Big(i\bv_1(H\q+2\bps\,\bq)-i\vb1(\bH\,\bq+2\psi\q)-i\bv_{\bar1}(H\bq+2\phi\q)+i v_1(\bH\q+2\bph\,\bq)\Big)\\
&+\frac{1}{8i}\Big(K(\be,e,e,n)\bv\bq+K(\be,e,\be,n)\bv\q+K(\be,e,\be,\bn)v\q+K(\be,e,e,\bn)v\bq\Big),
\end{split}
\eeq
\vskip -0.5cm
\beq\begin{split}
\frac{\p \q_{34}}{\p t}=&d_{m}B_2+\frac{1}{2}\Big(i\bv_1(H\q+2\bps\,\bq)-i\vb1(\bH\,\bq+2\psi\q)+i\bv_{\bar1}(H\bq+2\phi\q)-iv_1(\bH\q+2\bph\,\bq)\Big)\\
&+\frac{1}{8i}\Big(K(\bn,n,e,n)\bv\bq+K(\bn,n,\be,n)\bv\q+K(\bn,n,\be,\bn)v\q+K(\bn,n,e,\bn)v\bq\Big),
\end{split}
\eeq
\vskip -0.3cm
\beq
\begin{split}
~~~~~~~~\frac{\p H}{\p t}=&2\vba11-iB_2 H+\Big(\frac{1}{4}K(e,n,\be,n)+\frac{H^2}{2}+2\phi\bps\Big)\bv+\Big(\frac{1}{4}K(e,n,\be,\bn)+\frac{|H|^2}{2}+2|\psi|^2\Big)v,\\
=&2\v1b1-iB_2 H+\Big(\frac{1}{4}K(e,n,\be,n)+\frac{H^2}{2}+2\phi\bps\Big)\bv+\Big(\frac{1}{4}K(\be,n,e,\bn)+\frac{|H|^2}{2}+2|\phi|^2\Big)v,
\end{split}
\eeq
\vskip -0.3cm
\beq
~~~~~~~\frac{\p \bps}{\p t}=v_{\bar1\bar1}-i(2B_1+B_2)\bps+\Big(\frac{1}{8}K(e,n,e,n)+H\bps\Big)\bv+\Big(\frac{1}{8}\mathrm{Ric}(e, e)+\frac{H\bph}{2}+\frac{\bH\,\bps}{2}\Big)v,
\eeq
\vskip -0.5cm
\begin{equation}
\label{eq-vphi}
~~~~~~\frac{\p \phi}{\p t}=v_{11}+i(2B_1-B_2)\phi+\Big(\frac{1}{8}K(\be,n,\be,n)+H\phi\Big)\bv+\Big(\frac{1}{8}K(\be,n,\be,\bn)+\frac{H\psi}{2}+\frac{\bH\phi}{2}\Big)v,
\end{equation}
where $v_{11}, v_{1\bar1}, v_{\bar11}, v_{\bar1\bar1}$ are the covariant derivatives of $v_1$ and $v_{\bar1}$ as defined in \eqref{eq-conv2}. 

Now we are ready to calculate the second variation of the Willmore functional. Let $x: M^2\rightarrow N^4$ be a closed Willmore surface. It follows from \eqref{eq-1stvar} that
\beq\label{eq-2nd}
\begin{split}
\frac{\p^2}{\p t^2}\Big{|}_{t=0}\mathcal{W}(x)=2\,\mathfrak{Re}\int_{M^2}\Big\{&\frac{v}{4}\frac{\p }{\p t}\big(\psi \mathrm{Ric}(e, e)+\bph \mathrm{Ric}(\be, \be)+\bps K(\be,\bn,\be,\bn)+\phi K(e,\bn,e,\bn)\big)\\
&+v \frac{\p}{\p t}\big(\bH(|\phi|^2+|\psi|^2)+2H\bph\psi\big)+2v\big(\frac{\p \psi_{\bar1\bar1}}{\p t}+\frac{\p \bph_{11}}{\p t}\big)\Big\}dA_x.
\end{split}
\eeq

Using \eqref{eq-vKenebn}  $\sim$  \eqref{eq-vphi}, one can verify directly that 
\begin{equation}\label{eq-2vcurv}
\begin{split}
&\frac{\p }{\p t}\Big(\psi \mathrm{Ric}(e, e)+\bph \mathrm{Ric}(\be, \be)+\bps K(\be,\bn,\be,\bn)+\phi K(e,\bn,e,\bn)\Big)
\\
=&\psi \nabla_V \mathrm{Ric}(e, e)+\bph\nabla_V K (\be,\bn,\be,n)+\bps \nabla_V K(\be,\bn,\be,\bn)+\phi \nabla_V K(e,\bn,e,\bn)\\
&+ \bar{v}_{11} K\left(e,n,e,\bar{n}\right)+v_{11}  K\left(e,\bar{n},e,\bar{n}\right)+ \bar{v}_{\bar{1} \bar{1}}   K\left(\bar{e},\bar{n},\bar{e}, n\right)+
v_{\bar{1} \bar{1}}   K\left(\bar{e},\bar{n},\bar{e}, \bar{n}\right)\\
&+\bar{v}_1 \Big(2 \bar{\psi } \big(K\left(e,\bar{e},\bar{e},\bar{n}\right)+K\left(n,\bar{n},\bar{e},\bar{n}\right)\big)+\bar{\phi } \big(K\left(e,\bar{e},\bar{e},n\right)+ K\left(n,\bar{n},\bar{e},n\right)\big)\Big)\\
&+ \bar{v}_{\bar{1}} \Big(\psi  \big(K\left(n,\bar{n},e,n\right)-K\left(e,\bar{e},e,n\right)\big)-2 \phi \big( K\left(e,\bar{e},e,\bar{n}\right)-  K\left(n,\bar{n},e,\bar{n}\right)\big)\Big)\\
&+ v_1 \bar{\phi } \big(K\left(e,\bar{e},\bar{e},\bar{n}\right)-K\left(n,\bar{n},\bar{e},\bar{n}\right)\big)-v_{\bar{1}} \psi   \big(K\left(e,\bar{e},e,\bar{n}\right)+K\left(n,\bar{n},e,\bar{n}\right)\big)\\
&+\frac{1}{8} \bv 
\Big(8 H\big( \bar{\psi }   K\left(\bar{e},\bar{n},\bar{e}, \bar{n}\right)+\phi K\left(e,\bar{n},e,\bar{n}\right)\big)
+4 \big(\bar{H} \bar{\psi }+H \bar{\phi }\big)   K\left(\bar{e},\bar{n},\bar{e}, n\right)\\
&~~~~~~~~~~~~~+4\big(H \psi+\bar{H}\phi\big)
K\left(e,n,e,\bar{n}\right)
+2|K\left(e,n,e,\bar{n}\right)|^2+|K\left(e,\bar{n},e,\bar{n}\right)|^2+|K(e,n,e,n)|^2\Big)\\
&+\frac{1}{4} v
\Big(4   \bar{H}\big(\psi K\left(e,n,e,\bar{n}\right)+\bar{\phi }   K\left(\bar{e},\bar{n},\bar{e}, n\right)\big)+2 \left(\bar{H} \bar{\psi }+H \bar{\phi }\right)   K\left(\bar{e},\bar{n},\bar{e}, \bar{n}\right)\\
&~~~~~~~~~~~+ 2\left(
H \psi+\phi\bar{H}\right)K\left(e,\bar{n},e,\bar{n}\right)+K\left(e,n,e,\bar{n}\right)K\left(\bar{e},\bar{n},\bar{e}, \bar{n}\right)+K\left(e,\bar{n},e,\bar{n}\right)  K\left(\bar{e},\bar{n},\bar{e}, n\right)\Big)\\
&+i B_2  \Big(\psi  K\left(e,n,e,\bar{n}\right)+\bar{\psi }   K\left(\bar{e},\bar{n},\bar{e}, \bar{n}\right)+\phi  K\left(e,\bar{n},e,\bar{n}\right)+\bar{\phi
}   K\left(\bar{e},\bar{n},\bar{e}, n\right)\Big),
\end{split}
\end{equation}
and 
\allowdisplaybreaks
\begin{align}
&\frac{\p}{\p t}\Big(\bH(|\phi|^2+|\psi|^2)+2H\bph\psi\Big)\nonumber \\
= &
\left(\bar{v}_{\bar{1} 1}+\bar{v}_{1\bar{1}}\right) \big(\left| \psi \right| ^2+\left| \phi \right| ^2\big)+2 \left(v_{\bar{1} 1}+ v_{1\bar{1}}\right) \psi   \bar{\phi }
+\bar{v}_{\bar{1}
\bar{1}} \left(\bar{H}\phi+2 H \psi \right) \nonumber \\
&+\bar{v}_{11} \left(\bar{H} \bar{\psi }+2 H \bar{\phi }\right)+v_{\bar{1} \bar{1}} \bar{H} \psi  +v_{11} \bar{H} \bar{\phi } \nonumber \\
&+\frac{1}{8}\bv 
\Big(
K\left(e,n,e,\bar{n}\right)\big(\bar{H}\phi+2 H \psi \big)+K\left(\bar{e},\bar{n},\bar{e}, n\right)\left(\bar{H} \bar{\psi }+2 H \bar{\phi }\right)
+\bar{H}\psi K\left(e,n,e,n\right) \nonumber 
\\
&~~~~~~~~~~~~~+
\bar{H} \bar{\phi } K\left(\bar{e},n,\bar{e},n\right)+24 H^2 \psi  \bar{\phi }
+8 \bar{H}^2 \phi \bar{\psi }+24 \left| H\right| ^2 \left(\left| \psi \right|
^2+\left| \phi \right| ^2\right)+8 \big(\left| \psi \right| ^4+\left| \phi \right| ^4\big) \nonumber 
\\
&~~~~~~~~~~~~~+48 \left| \phi \right| ^2\left| \psi \right| ^2+\big(\left| \psi \right| ^2+\left|
\phi \right| ^2 \big)\big(K\left(e,n,\bar{e},\bar{n}\right)+K\left(e,\bar{n},\bar{e},n\right)\big)
 +4
\psi  \bar{\phi } K\left(e,n,\bar{e},n\right)
\Big) \label{eq-2vmean}
\\
&+\frac{1}{8} v \Big(
  K\left(e,\bar{n},e,\bar{n}\right)\big(  \bar{H}\phi+2 H \psi \big)+ K\left(\bar{e},\bar{n},\bar{e},\bar{n}\right)\big(\bar{H} \bar{\psi }+2
H \bar{\phi }\big)+
 \bar{H}\psi K\left(e,n,e,\bar{n}\right) \nonumber
\\
&~~~~~~~~~~~+\bar{H} \bar{\phi } K\left(\bar{e},\bar{n},\bar{e}, n\right)+16 \bar{H}^2 \big(\left| \phi \right| ^2+\left| \psi \right| ^2\big)+48 \left| H\right| ^2\psi  \bar{\phi }+
2 \big(\left| \phi \right| ^2+\left| \psi \right| ^2\big) K\left(e,\bar{n},\bar{e},\bar{n}\right) \nonumber
 \\
&~~~~~~~~~~~+2 \psi  \bar{\phi } \big(K\left(e,n,\bar{e},\bar{n}\right)+
K\left(e,\bar{n},\bar{e},n\right)\big)+32 \psi  \bar{\phi }\big(
|\bar{\phi }|^2+|\bar{\psi }|^2\big)\Big) \nonumber \\
&+i B_2 \Big(\bar{H} \left| \psi \right| ^2+\bar{H} \left| \phi \right| ^2+2 H \psi  \bar{\phi }\Big). \nonumber
\end{align}

\def\psib1{\psi_{\bar1}}
\def\psbb1{\psi_{\bar1\bar1}}
\def\psb11{\psi_{\bar11}}
\begin{lemma}
Modulo some globally defined exact $2$-forms, we have 
\begin{equation}\label{eq-itm1n}
\begin{split}
v\frac{\p \psi_{\bar1\bar1}}{\p t}\bq\wedge\q\equiv &(v\psi_{\bar1\bar1}-v_{\bar1}\psi_{\bar1})\frac{\p }{\p t}(\q\wedge\bq)+v_{\bar1\bar1}\frac{\p \psi}{\p t}\bq\wedge\q+(vD\psi_{\bar1}+\psi Dv_{\bar1})\wedge\frac{\p \q}{\p t}\\
&+i(v_{\bar1}\psi-v\psi_{\bar1})(\frac{\p \q_{12}}{\p t}+\frac{\p \q_{34}}{\p t})\wedge\q, 
\end{split}
\end{equation}
\begin{equation}\label{eq-phi22}
\begin{split}
v\frac{\p \bph_{11}}{\p t}\q\wedge\bq\equiv &(v_1\baph1-v\bbaph11)\frac{\p }{\p t}(\q\wedge\bq)+v_{11}\frac{\p \bph}{\p t}\q\wedge\bq+(vD\bph_{1}-v_{1}D\bph)\wedge\frac{\p \bq}{\p t}~~\\
&+i(v\bph_{1}-v_{1}\bph)(\frac{\p \q_{12}}{\p t}-\frac{\p \q_{34}}{\p t})\wedge\bq. ~~
\end{split}
\end{equation}
\end{lemma}
\begin{proof}
It follows from the definition \eqref{eq-2ndpsi} that
\beq
\begin{split}
\frac{\p \psi_{\bar1\bar1}}{\p t}\bq\wedge\q=&\frac{\p}{\p t}(D\psi_1\wedge \q)-\psi_{\bar1\bar1}\frac{\p }{\p t}(\q\wedge\bq)=D\psi_{\bar1}\wedge\frac{\p \q}{\p t}+\frac{\p}{\p t} (D \psi_{\bar1})\wedge\q-\psi_{\bar1\bar1}\frac{\p }{\p t}(\q\wedge\bq)\\
=&\psi_{\bar1\bar1}\frac{\p }{\p t}(\q\wedge\bq)+D\psi_{\bar1}\wedge\frac{\p \q}{\p t}+D\frac{\p \psi_{\bar1}}{\p t}\wedge\q-i\psi_{\bar1}(\frac{\p \q_{12}}{\p t}+\frac{\p \q_{34}}{\p t})\wedge\q, 
\end{split}
\eeq
where 
$ D\frac{\p \psi_{\bar1}}{\p t}
$ is the covariant differential of $\frac{\p \psi_{\bar1}}{\p t}$ defined as follows 
$$D\frac{\p \psi_{\bar1}}{\p t}=d\frac{\p \psi_{\bar1}}{\p t}-i(\q_{12}+\q_{34})\frac{\p \psi_{\bar1}}{\p t}.$$
As in the proof of Lemma~\ref{lem-int1}, we can derive that 
$$D\frac{\p \psi_{\bar1}}{\p t}\wedge (v\q)=-Dv\wedge(\frac{\p \psi_{\bar1}}{\p t}\q)+d_m(v\frac{\p \psi_{\bar1}}{\p t}\q)=-v_{\bar1}\frac{\p \psi_{\bar1}}{\p t} \bq\wedge\q+d_m(v\frac{\p \psi_{\bar1}}{\p t}\q),$$
where $\frac{\p \psi_{\bar1}}{\p t} \bq\wedge\q$ can be calculated as below 
$$\frac{\p \psi_{\bar1}}{\p t}\bq\wedge\q=\psi_{\bar1}\frac{\p }{\p t}(\q\wedge\bq)+D\psi\wedge\frac{\p \q}{\p t}+D\frac{\p \psi}{\p t}\wedge\q-i\psi(2\frac{\p \q_{12}}{\p t}+\frac{\p \q_{34}}{\p t})\wedge\q.$$
Note that 
$$D\frac{\p \psi}{\p t}\wedge (v_{\bar1}\q)=-Dv_{\bar1}\wedge(\frac{\p \psi}{\p t}\q)+d_m(v_{\bar1}\frac{\p \psi}{\p t}\q).$$
Hence we have 
\begin{equation}\label{eq-itm1}
\begin{split}
v\frac{\p \psi_{\bar1\bar1}}{\p t}\bq\wedge\q\equiv &(v\psi_{\bar1\bar1}-v_{\bar1}\psi_{\bar1})\frac{\p }{\p t}(\q\wedge\bq)+v_{\bar1\bar1}\frac{\p \psi}{\p t}\bq\wedge\q+(vD\psi_{\bar1}-v_{\bar1}D\psi)\wedge\frac{\p \q}{\p t}\\
&-iv\psi_{\bar1}(\frac{\p \q_{12}}{\p t}+\frac{\p \q_{34}}{\p t})\wedge\q+iv_{\bar1}\psi(2\frac{\p \q_{12}}{\p t}+\frac{\p \q_{34}}{\p t})\wedge\q,
\end{split}
\end{equation}
modulo those exact forms, which can be verified directly to be globally well-defined. 
Now  \eqref{eq-itm1n} follows from 
$$\vb1\frac{\p \theta}{\p t}\wedge D\psi=-d(\vb1\frac{\p \q}{\p t}\psi)+\psi D \vb1\wedge \frac{\p \q}{\p t}-i \vb1 \psi \frac{\p \q_{12}}{\p t}\wedge\q.$$
Similarly, we can prove \eqref{eq-phi22}. 
\end{proof}

It follows from the above lemma, \eqref{eq-pqt} and \eqref{eq-vthe12}  $\sim$  \eqref{eq-vphi} that the expressions of $\frac{\p \psi_{\bar1\bar1}}{\p t}$ and $\frac{\p \bph_{11}}{\p t}$ are given as follows: 
\begin{equation}\label{eq-2vpsi}
\begin{split}
v\frac{\p \psi_{\bar1\bar1}}{\p t}\equiv&\,\frac{1}{8}v_{\bar{1} \bar{1}} \Big(v K\left(\bar{e},\bar{n},\bar{e},\bar{n}\right)+\bar{v} K\left(\bar{e},\bar{n},\bar{e},n\right)+4
v \psi  \bar{H}+4 \phi  \bar{H} \bar{v}\Big)+|\bar{v}_{11}|^2+v_{\bar{1} 1} \left(\bar{v} \left| \psi \right| ^2+v \psi  \bar{\phi }\right)\\
&\,+v_{\bar{1}}{}^2\psi
 \bar{H}-2 |v_{\bar{1}}|^2 \left| \psi \right| ^2- v_{\bar{1}}\psi _{\bar{1}} \left(2v\bar{H} +\bar{v} H \right)+2 v \bar{v}_1 \bar{\psi }\psi _{\bar{1}}\\
&\,-\frac{1}{8}v_{\bar{1}}\psi \Big(v   K\left(e,\bar{e},e,\bar{n}\right)+v   K\left(n,\bar{n},e,\bar{n}\right)+  \bar{v} K\left(e,\bar{e},e,n\right)+
 \bar{v} K\left(n,\bar{n},e,n\right) \Big)\\
&\,+\frac{1}{8}|v|^2 \Big(\psi _{\bar{1}} \left( K\left(e,\bar{e},e,n\right)+ K\left(n,\bar{n},e,n\right)\right)+4 \psi
_{\bar{1} \bar{1}} H+8  \psi _{\bar{1} 1} \bar{\psi }\Big)\\
&\,+\frac{1}{8}v^2 \Big(\psi _{\bar{1}} \left(K\left(e,\bar{e},e,\bar{n}\right)+K\left(n,\bar{n},e,\bar{n}\right)\right)+4 \bar{H} \psi _{\bar{1} \bar{1}}+8
\psi _{\bar{1} 1} \bar{\phi }\Big)+i B_2 v \psi _{\bar{1} \bar{1}},
 \end{split}
\end{equation}
\begin{equation}\label{eq-2vphi}
\begin{split}
v\frac{\p \bph_{11}}{\p t}\equiv&\frac{1}{8} v_{11} \Big(v K\left(e,\bar{n},e,\bar{n}\right)+\bar{v} K\left(e,n,e,\bar{n}\right)+4 \bar{H} \bar{v}
\bar{\psi }+4 v \bar{H} \bar{\phi }\Big)+ |\bar{v}_{\bar{1} \bar{1}}|^2+ v_{1\bar{1}} \left(\bar{v} \left| \phi \right| ^2+v \psi  \bar{\phi }\right)\\
&+
v_1{}^2 \bar{H} \bar{\phi }-2 |v_1|^2 \left| \phi \right| ^2- v_1\bar{\phi }_1
\left(2 v \bar{H}+H \bar{v}\right)+2v \bar{v}_{\bar{1}} \phi\bar{\phi }_1 \\
&+\frac{1}{8}v_1\bar{\phi } \Big(v   K\left(e,\bar{e},\bar{e},\bar{n}\right)-v   K\left(n,\bar{n},\bar{e},\bar{n}\right)+\bar{v}  
K\left(e,\bar{e},\bar{e},n\right)-\bar{v}   K\left(n,\bar{n},\bar{e},n\right)\Big)\\
&+\frac{1}{8}|v|^2 \Big(\bar{\phi }_1 \left( K\left(n,\bar{n},\bar{e},n\right) -K\left(e,\bar{e},\bar{e},n\right)\right)+4
\bar\phi _{11} H +8 \phi    \bar{\phi }_{1\bar{1}}\Big)\\
&+\frac{1}{8}v^2 \Big(\bar{\phi }_1 \left(K\left(n,\bar{n},\bar{e},\bar{n}\right)-K\left(e,\bar{e},\bar{e},\bar{n}\right)\right)+4 \bar\phi _{11} \bar{H}+8
\psi  \bar{\phi }_{1\bar{1}}\Big)+i B_2 v \bar\phi _{11},
\end{split}
\end{equation}
where the notation $"\equiv"$ means equaling in the sense of integration. 

Substituting \eqref{eq-2vcurv}  $\sim$  \eqref{eq-2vphi} into \eqref{eq-2nd}, we can derive that 
\allowdisplaybreaks
\begin{align}
&\frac{\p^2}{\p t^2}\Big{|}_{t=0}\mathcal{W}(x) \nonumber \\
=&2\,\mathfrak{Re}\int_{M^2}\Bigg\{v_{11} \left(\frac{1}{2} K\left(e,\bar{n},e,V\right)+\bar{H} \left(v\bar{\phi }+ \bar{v}\bar{\psi
}\right)+\frac{1}{2} \bar{v} \left(H \bar{\phi }+\bar{H} \bar{\psi
}\right)\right)\nonumber \\
+&v_{\bar{1}\bar{1}}\left(\frac{1}{2} K\left(\bar{e},\bar{n},\bar{e},V\right)+\bar{H}(v \psi +\bar{v}\phi)+\frac{1}{2} \bar{v} \left(\bar{H}\phi +H
\psi \right)\right)\nonumber \\
+&v_{1\bar{1}} \left(\frac{1}{2} \bar{v} \left(\left| \psi \right| ^2+3 \left| \phi \right| ^2\right)+2 v \psi  \bar{\phi }\right)
+v_{\bar{1}1} \left(\frac{1}{2} \bar{v} \left(\left| \phi \right| ^2+3 \left| \psi \right| ^2\right)+2 v \psi  \bar{\phi }\right)\nonumber \\
+&v_1 \Bigg( \left(\frac{v \bar{\phi }}{2} +\frac{\bar{v}\bar{\psi }  }{4} \right)\text{Ric}\left(\bar{e},\bar{n}\right)+\frac{3\bar{v} \bar{\phi } }{8} \Big(K\left(e,\bar{e},\bar{e},n\right)-K\left(n,\bar{n},\bar{e},n\right)\Big)\nonumber \\
&~~~~~~-2\psi _{\bar{1}} \left( v
\bar{\phi }+\bar{v} \bar{\psi }\right)+ \bar{v} \left(\bar{\phi }\phi _{\bar{1}}+\bar{\psi }\psi _{\bar{1}}\right)-\bar{\phi }_1 H \bar{v}+v_1\bar{H}\bar{\phi}\Bigg)\nonumber \\
+&v_{\bar{1}} \Bigg(\left( \frac{v\psi }{2}+\frac{ \bar{v} \phi }{4}\right)\text{Ric}\left(e,\bar{n}\right)-\frac{3 \bar{v} \psi }{8}\Big(K\left(e,\bar{e},e,\bar n\right)+K\left(n,\bar{n},e,n\right)\Big)\nonumber \\
&~~~~~~-2\bar{\phi }_1 \left(\phi  \bar{v}+v \psi
\right)+\bar{v}\left(\phi\bar{\phi }_1+\psi  \bar{\psi }_1\right) -\psi _{\bar{1}} H \bar{v}+v_{\bar 1}\bar{H}\psi\Bigg)\nonumber \\
+&{\frac{ v^2}{32} }
\Bigg(4 \psi  \bar{\phi } \Big(K\left(e,n,\bar{e},\bar{n}\right)+K\left(e,\bar{n},\bar{e},n\right)\Big)+\left(4 \text{Ric}\left(\bar{n},\bar{n}\right)+64 \psi  \bar{\phi }+8 \bar{H}^2\right)\left(\left| \psi \right| ^2+\left|
\phi \right| ^2\right) \label{eq-original} \\
&~~~~~-2 \bar{H} \Big(\bar{\psi } K\left(\bar{e},\bar{n},\bar{e},\bar{n}\right)+\phi  K\left(e,\bar{n},e,\bar{n}\right)\Big)+6H \Big(\psi  K\left(e,\bar{n},e,\bar{n}\right)+
\bar{\phi } K\left(\bar{e},\bar{n},\bar{e},\bar{n}\right)\Big)\nonumber \\
&~~~~~+\text{Ric}\left(\bar{e},\bar{e}\right) K\left(e,\bar{n},e,\bar{n}\right)+
\text{Ric}(e,e) K\left(\bar{e},\bar{n},\bar{e},\bar{n}\right)+48 \psi  \bar{\phi } \left| H\right| ^2+64{\bar{\phi }_1 \psi _{\bar{1}}}\nonumber \\
&~~~~~+4\bar{\phi }_1\Big(K\left(e,\bar{e},\bar{e},\bar{n}\right)+K\left(n,\bar{n},\bar{e},\bar{n}\right)\Big)+4\psi
_{\bar{1}} \Big(K\left(n,\bar{n},e,\bar{n}\right)-K\left(e,\bar{e},e,\bar{n}\right)\Big)\Bigg)\nonumber \\
+&\frac{ \left| v\right| ^2 }{32}\Bigg(
{8 \left(\psi  \bar{\phi } \text{Ric}(n,n)+\phi  \bar{\psi } \text{Ric}(\bar n,\bar n)\right)+ 4\Big(K\left(e,n,\bar{e},\bar{n}\right)+K\left(e,\bar{n},\bar{e},n\right)\Big)\left(\left| \psi \right| ^2+\left| \phi \right| ^2\right)}\nonumber \\
&~~~~~+4H \Big(\bar{\psi } K\left(\bar{e},\bar{n},\bar{e},\bar{n}\right)+\phi  K\left(e,\bar{n},e,\bar{n}\right)+2
\bar{\phi } \text{Ric}\left(\bar{e},\bar{e}\right)+2\psi  \text{Ric}(e,e)\Big)\nonumber \\
&~~~~~+4\bar{H} \Big(\psi  K(e,n,e,n)+
\bar{\phi } K\left(\bar{e},n,\bar{e},n\right)+2 \phi  \text{Ric}(e,e)+2 \bar{\psi } \text{Ric}\left(\bar{e},\bar{e}\right)\Big)\nonumber \\
&~~~~~+48\left(H^2 \psi  \bar{\phi }+\bar{H}^2\phi   \bar{\psi }\right)+80\left| H\right|
^2 \left(\left| \psi \right| ^2+\left| \phi \right| ^2\right)
+192 \left| \psi \right| ^2 \left| \phi \right| ^2+32 \left(\left| \psi \right| ^4+\left| \phi \right| ^4\right)\nonumber \\
&~~~~~+\left| K\left(e,\bar{n},e,\bar{n}\right)\right|
^2+\left| K(e,n,e,n)\right| ^2+2 \left| \text{Ric}(e,e)\right| ^2-64\left( \left| \psi _{\bar{1}}\right| ^2+ \left| \phi _{\bar{1}}\right|
^2)\right.\nonumber \\
&~~~~~+4\psi _{\bar{1}} \Big(K\left(e,\bar{e},e,n\right)+K\left(n,\bar{n},e,n\right)\Big)-4\bar{\psi
}_1\Big(K\left(e,\bar{e},\bar{e},\bar{n}\right)+K\left(n,\bar{n},\bar{e},\bar{n}\right)\Big)\nonumber \\
&~~~~~+4\phi _{\bar{1}} \Big(K\left(e,\bar{e},e,\bar{n}\right)-K\left(n,\bar{n},e,\bar{n}\right)\Big)-4\bar{\phi }_1 \Big(
K\left(e,\bar{e},\bar{e},n\right)-K\left(n,\bar{n},\bar{e},n\right)\Big)\Bigg)\nonumber \\
+&2 \left(\left| v_{11}\right| ^2+\left | v_{\bar 1 \bar1}\right | ^2\right)-4 \left(\left| v_{\bar{1}}\right| ^2 \left| \psi \right| ^2+\left| v_1\right| ^2 \left| \phi \right| ^2\right)\nonumber \\
+&\frac{v}{4}\Bigg[(\nabla_V K)(e,n,e,\bar{n})\psi+(\nabla_V K)(\bar{e},\bar{n},\bar{e},\bar{n})\bar{\psi}+(\nabla_V K)(\bar{e},\bar{n},\bar{e},n)\bar{\phi}+(\nabla_V K)(e,\bar{n},e,\bar{n})\phi\Bigg]\Bigg\}dA_x, \nonumber 
\end{align}
where we have used the Codazzi equations \eqref{eq-Codazzi} to convert the covariant derivatives of $H$ to that of $\phi$ and $\psi$, and 
 $$\int_{M^2}v^2 \bar{\phi} \psi_{\bar 1 1}dA_x=-\int_{M^2}\left(v^2 \bar{\phi}_1 \psi_{\bar 1}+2v v_{\bar{1}} \bar{\phi}\psi_{\bar 1}\right)d A_x,$$
 $$\int_{M^2}v^2  \psi  \bar{\phi}_{1\bar 1 } dA_x=-\int_{M^2}\left(v^2  \psi_{\bar 1}  \bar{\phi}_1+2v v_{\bar{1}}  \psi  \bar{\phi}_{1}\right)d A_x,$$
 $$\int_{M^2}|v|^2  \bps  {\psi}_{\bar 1 1} dA_x=-\int_{M^2}\left(\left(v_{ 1} \bv+v\bv_{ 1} \right)\bps \psi_{\bar 1}+|v|^2|\psi_{\bar 1}|^2\right)d A_x,$$
 $$\int_{M^2}|v|^2  \phi  {\bph}_{1\bar 1} dA_x=-\int_{M^2}\left(\left(v_{ 1} \bv+v\bv_{ 1} \right)\phi_{\bar 1} \bph_{ 1}+|v|^2|\phi_{\bar 1}|^2\right)d A_x$$
 to deal with those terms containing the second order covariant derivatives of $\phi$ and $\psi$. 

Next, we simplify the expression of $\frac{\p^2}{\p t^2}\Big{|}_{t=0}\mathcal{W}(x)$. Set 
$$Ric(e)^\perp\triangleq\frac{1}{2}\big(\mathrm{Ric}(e,\bn)n+\mathrm{Ric}(e,n)\bn\big),$$
which is just 
the normal projection of $Ric(e)$;  
and 
$$\bll\alpha_1\otimes\beta_1,\alpha_2\otimes\beta_2\Brr\triangleq \bl\alpha_1\otimes\beta_1,\alpha_2\otimes\beta_2\Br+\bl\alpha_1\wedge\alpha_2, \beta_1\wedge\beta_2\Br,$$
which can be seen as a twisted bilinear inner metric defined on the vector bundle $NM^2\otimes NM^2$. 

It follows from the definition of $V, \vec H, B(e,e)$, and \eqref{eq-hessV} that  some coefficients of the terms  $\vaa$ and $\vbb$  in  \eqref{eq-original} can be collected as follows, 
\begin{equation}\begin{split}\label{eq-v11sim}
v_{11}K\left(e,\bar{n},e,V\right)+\bvaa K\left(e,n,e,V\right)&=\frac{1}{2}K(e,\He V(\Be,\Be),e,V)\\
&=\frac{1}{2}W(e,\He V(\Be,\Be),e,V)+\frac{1}{4}\bl\He V(\Be,\Be),V\Br \mathrm{Ric}(e,e),
\end{split}\end{equation}
\begin{equation}\begin{split}
(v_{11}\bar{v}+\bvaa v) \left(\bar{H} \bar{\psi }+H \bar{\phi }\right)&=\frac{1}{4} \bl \He V(\Be,\Be),V\Br \bl\B,\vH\Br,~~~~~~~~~~~~~~~~~~~~~~~~~~~~~~~~~~~~~~~~~~
\end{split}\end{equation}
\begin{equation}\begin{split}
(v_{11}\bar{H}+\bvaa H) \left(\bar{v} \bar{\psi }+v \bar{\phi }\right)&=\frac{1}{4} \bl \He V(\Be,\Be),\vH\Br \bl\B,V\Br,~~~~~~~~~~~~~~~~~~~~~~~~~~~~~~~~~~~~~~~~~~
\end{split}\end{equation}
\begin{equation}\begin{split}\label{eq-HessVN}
\left| v_{11}\right|^2+\left | v_{\bar 1\bar 1}\right |^2&=\frac{1}{8}\left|\He V(e,e)\right|^2.~~~~~~~~~~~~~~~~~~~~~~~~~~~~~~~~~~~~~~~~~~~~~~~~~~~~~~~~~~~~~~~~~~~~~~~~~~~~
\end{split}\end{equation} 
By the definition of $\ca(V)$ and \eqref{eq-hessV2}, we can rewrite the terms $\vab$ and $v_{\bar 1 1}$  in \eqref{eq-original} as below, 
\begin{equation}\begin{split}\label{eq-LplV}
&\mathfrak{Re}\Big[v_{1\bar{1}} \left(\frac{1}{2} \bar{v} \left(\left| \psi \right| ^2+3 \left| \phi \right| ^2\right)+2 v \psi  \bar{\phi }\right)
+v_{\bar{1}1} \left(\frac{1}{2} \bar{v} \left(\left| \phi \right| ^2+3 \left| \psi \right| ^2\right)+2 v \psi  \bar{\phi }\right)\Big]\\
=&\frac{1}{2}\big\langle \Delta^\perp V, \ca(V)\big\rangle-\frac{1}{8}\big\langle R_{e\bar e}^{\perp}V\wedge V, B(e,e)\wedge B(\bar e,\bar e)\big\rangle.
\end{split}\end{equation}
For the terms $v_1$ and $v_{\bar 1}$ in \eqref{eq-original}, using \eqref{eq-Weyl} and the fact that Weyl curvature tensor is trace-free, we have  
\begin{equation}\begin{split}
&K\left(e,\bar{e},\bar{e},n\right)-K\left(n,\bar{n},\bar{e},n\right)=W\left(e,\bar{e},\bar{e},n\right)-W\left(n,\bar{n},\bar{e},n\right)=2W\left(e,\bar{e},\bar{e},n\right),~~~~~~~~~~~~~~
 \end{split}\end{equation}
\begin{equation}\begin{split}
K\left(e,\bar{e},e, n\right)+K\left(n,\bar{n},e,n\right)=W\left(e,\bar{e},e, n\right)+W\left(n,\bar{n},e,n\right)=2W\left(e,\bar{e},e, n\right), ~~~~~~~~~~~~~~
 \end{split}\end{equation}
\begin{equation}\begin{split}
&v_1\bar{v} \bar{\phi } W\left(e,\bar{e},\bar{e},n\right)+\bar v_1v \bar{\psi } W\left(e,\bar{e},\bar{e},\bn\right)=\frac{1}{4}\bll\nb V\otimes \B, V\otimes W^\perp(e,\Be,\Be) \Brr, ~~~~~~~~~~~~~~
 \end{split}\end{equation}
\begin{equation}\begin{split}
&v_1 \Bigg( \frac{v \bar{\phi }}{2} \text{Ric}\left(\bar{e},\bar{n}\right)+\frac{\bar{v}\bar{\psi }  }{4} \text{Ric}\left(\bar{e},\bar{n}\right)\Bigg)
+\bar v_1 \Bigg( \frac{v \bar{\phi }}{4} \text{Ric}\left(\bar{e},n\right)+\frac{\bar{v}\bar{\psi }  }{2} \text{Ric}\left(\bar{e},n\right)\Bigg),\\
&~~~~~~~~~~~~~~~~~~=\frac{1}{8}\bl \nb V\otimes \B,Ric(\Be)^\perp\otimes V\Br+\frac{1}{16}\bll \nb V\otimes \B, Ric(\Be)^\perp\otimes V\Brr,
 \end{split}\end{equation}
\begin{equation}\begin{split}
&2\left(v_1\psi _{\bar{1}}+ \bar v_1\Ga\right)\left( v
\bar{\phi }+\bar{v} \bar{\psi }\right) - \left(v_1\bar{v}+\bar v_1 v\right) \left(\bar{\phi }\phi _{\bar{1}}+\bar{\psi }\psi _{\bar{1}}\right)~~~~~~~~~~\\
&~~~~~~~~~~~~~~~~~~ =\frac{1}{16}\bl \nb V\otimes \B,2\nabla_e B(\be,\be)\otimes V-V\otimes \nabla_e \Bb\Br,~~~~~~~~~~~~~~~~~~~
 \end{split}\end{equation}
\begin{equation}\begin{split}
&v_1\bar{\phi }_1 H \bar{v}+v_{\bar 1}\psi _{\bar{1}} H \bar{v}+\bar v_1\bar{\psi }_1 \bH {v}+\bar v_{\bar 1}\phi _{\bar{1}} \bH {v}=\frac{1}{4}\mathfrak{Re}\bll \nb \B\otimes \nb V, \vH\otimes V\Brr,~~~~~~~~~~~~~~
\end{split}\end{equation}
\begin{equation}\begin{split}
&v_1^2\bar{H}\bar{\phi}+v_{\bar 1}^2\bar{H}\psi+\bar v_{\bar1}^2{H}{\phi}+\bar v_{ 1}^2{H}\bar\psi=\frac{1}{4}\mathfrak{Re}\bll \nb V\otimes \B, \vH\otimes \nb V\Brr,~~~~~~~~~~~~~~~~~~~~~~~~
\end{split}\end{equation}
\begin{equation}\begin{split}\label{eq-vb1psi}
~~~~~~\left| v_{\bar{1}}\right| ^2 \left| \psi \right| ^2+\left| v_1\right| ^2 \left| \phi \right| ^2=\frac{1}{32}\bll \nabla_e^{\perp}V\otimes B(\bar e,\bar e),  \nb V\otimes\B\Brr.~~~~~~~~~~~~~~~~~~~~~~~~~~~~~~~~~~~
\end{split}\end{equation}
Similarly, we can use 
$$K\left(e,\bar{n},e,\bar{n}\right)=W\left(e,\bar{n},e,\bar{n}\right),~~~K\left(\bar{e},\bar{n},\bar{e},\bar{n}\right)=W\left(\bar{e},\bar{n},\bar{e},\bar{n}\right),~~~W\left(\bar{e},\bar{n},\bar{e},n\right)=0, ~~~W\left({e},{n},\bar{e},n\right)=0,$$
to simplify the terms $v^2$ and $|v|^2$ in \eqref{eq-original} as follows,  
\begin{equation}\begin{split}
&\mathfrak{Re}\Big[\bar{H} \Big(\bar{\psi } W\left(\bar{e},\bar{n},\bar{e},\bar{n}\right)+\phi  W\left(e,\bar{n},e,\bar{n}\right)\Big)+H \Big(\psi  W\left(e,\bar{n},e,\bar{n}\right)+
\bar{\phi } W\left(\bar{e},\bar{n},\bar{e},\bar{n}\right)\Big)\Big]~~~~~~~~\\
&~~~~~~~~~~~~~~~~~~~~~~~~~~~~~~~~~~~~~~~~~~~~~~~~~~~~~~~~~~~=2\mathfrak{Re}\Big(W\big(\Be,V,\Be,  V\big)\bl\B, \vH\Br\Big), 
\end{split}\end{equation}
\begin{equation}\begin{split}
&2\mathfrak{Re}\Big[v^2H \Big(\psi  W\left(e,\bar{n},e,\bar{n}\right)+
\bar{\phi } W\left(\bar{e},\bar{n},\bar{e},\bar{n}\right)\Big)\Big]+|v|^2H \Big(\bar{\psi } W\left(\bar{e},\bar{n},\bar{e},\bar{n}\right)+\phi  W\left(e,\bar{n},e,\bar{n}\right)\Big)\\
&~~~~~~~~~~~~~+|v|^2\bar{H} \Big(\psi  W(e,n,e,n)
+\bar{\phi } W\left(\bar{e},n,\bar{e},n\right)\Big)
=4\mathfrak{Re}\Big(W\big(\Be,V,\Be,  \vH\big)\bl\B, V\Br\Big), 
\end{split}\end{equation}
\begin{equation}\begin{split}
&|v|^2H \Big(\bar{\psi } W\left(\bar{e},\bar{n},\bar{e},\bar{n}\right)+\phi  W\left(e,\bar{n},e,\bar{n}\right)\Big)
+|v|^2\bar{H} \Big(\psi  W(e,n,e,n)
+\bar{\phi } W\left(\bar{e},n,\bar{e},n\right)\Big)~~~~~~~\\
&~~~~~~~~~~~~~~~~~~~~~~~~~~~~~~~~~~~~~~~~~~~~~~~~~~~~~~~~~~~~~~~=2\mathfrak{Re}\Big(\left|V\right|^2W\big(\Be, \vH,\Be, \B\big)\Big), 
\end{split}\end{equation}
\begin{equation}\begin{split}
&H\Big(\bar{\phi } \text{Ric}\left(\bar{e},\bar{e}\right)+\psi  \text{Ric}(e,e)\Big)+\bH\Big( \phi  \text{Ric}(e,e)+ \bar{\psi } \text{Ric}\left(\bar{e},\bar{e}\right)\Big)~~~~~~~~~~~~~~~~~~~~~~~~~~~~~~~~~~~~~~\\
&~~~~~~~~~=\frac{1}{2}\Big(\text{Ric}(\Be,\Be)\bl\B,\vH\Br+\text{Ric}(e,e)\bl\Bb,\vH\Br\Big), 
\end{split}\end{equation}
\begin{equation}\begin{split}
~~~~~~2\mathfrak{Re}\Big(v^2(\bar{H}^2\left(\left| \psi \right| ^2+\left|\phi \right| ^2\right)+6 \psi  \bar{\phi } \left| H\right| ^2)\Big)
+|v|^2\Big(6\left(H^2 \psi  \bar{\phi }+\bar{H}^2\phi   \bar{\psi }\right)+10\left| H\right|
^2 \left(\left| \psi \right| ^2+\left| \phi \right| ^2\right)\Big)\\
={6}\bl \vH,V\Br \bl \ca(V),\vH\Br+\left|  B(e,e) \right|^2\left|\vH\wedge V\right|^2~~~~~~~~~~~~~~~~~~~~~~~~~\\
=3\Big(|\vH|^2\bl \ca(V),V\Br+|V|^2\bl \ca(\vH),\vH\Br\Big)-\frac{1}{2}\left|  B(e,e) \right|^2\left|\vH\wedge V\right|^2,\!\!\!
\end{split}\end{equation}
\begin{equation}\begin{split}
v^2\text{Ric}\left({e},{e}\right) W\left(\Be,\bar{n},\Be,\bar{n}\right)+\bar v^2\text{Ric}\left({e},{e}\right) W\left(\Be,{n},\Be,{n}\right)=4\text{Ric}\left({e},{e}\right) W\left(\Be,V,\Be,V\right),~~~~~~~~~~~~~\;
\end{split}\end{equation}
\begin{equation}\begin{split}
v^2{\bar{\phi }_1 \psi _{\bar{1}}}+\bar v^2\phi_{\bar 1} \bar\psi _{1}-\left| v\right|^2\left( \left| \psi _{\bar{1}}\right|^2+ \left| \phi _{\bar{1}}\right|
^2\right)=-\frac{1}{16}\left | V\wedge \nb B(e,e) \right|^2,~~~~~~~~~~~~~~~~~~~~~~~~~~~~~~\;
\end{split}\end{equation}
\begin{equation}\begin{split}
~~v^2\bar{\phi }_1W\left(e,\bar{e},\bar{e},\bar{n}\right)+\bar v^2 \bar\psi
_{{1}} W\left(e,\bar{e},\Be,{n}\right)-|v|^2\bar{\phi }_1 W\left(e,\bar{e},\bar{e},n\right)-|v|^2\bar{\psi
}_1W\left(e,\bar{e},\bar{e},\bar{n}\right)&~~~~~~~~~~~~~~\\
=\frac{1}{2}\bl W^\perp(e,\Be,\Be)\wedge V, V\wedge\nb\B\Br&,
\end{split}\end{equation}
\begin{equation}\begin{split}
&\left(2|v|^2\bar\psi  {\phi }+v^2(\left| \psi \right| ^2+\left|
\phi \right| ^2)\right)\text{Ric}\left(\bar{n},\bar{n}\right)+\Big( 2|v|^2\psi  \bar{\phi }+ \bar v^2(\left| \psi \right| ^2+\left|
\phi \right| ^2)\Big)\text{Ric}\left({n},{n}\right),\\
&~+\left(v^2 \psi  \bar{\phi } +\bar v^2 \bar\psi  {\phi }+|v|^2(\left| \psi \right| ^2+\left|
\phi \right| ^2) \right)\Big(K\left(e,n,\bar{e},\bar{n}\right)+K\left(e,\bar{n},\bar{e},n\right)\Big)
=2\bl \ric(V), \ca(V)\Br, 
\end{split}\end{equation}
\begin{equation}\begin{split}\label{eq-2cav}
2\left(v^2 \psi  \bar{\phi } +\bar v^2 \bar\psi  {\phi }\right)\left(\left| \psi \right| ^2+\left|
\phi \right| ^2\right)
+|v|^2\left(\left| \psi \right| ^4+\left|
\phi \right| ^4+6\left| \phi \right| ^2\left| \psi \right| ^2\right)
=\frac{1}{4}\bl \ca(V), \ca(V)\Br.~~~~~~
\end{split}\end{equation}
In conclusion, 
using \eqref{eq-v11sim}  $\sim$  \eqref{eq-2cav}, we can obtain the simplified second variation formula of $\mathcal{W}$. 
\begin{theorem}\label{thm-2ndvar}
Let $x:M^2\rightarrow (N^4, [g])$ be a closed Willmore surface, then 
 $\frac{\p^2}{\p t^2}\Big{|}_{t=0}\mathcal{W}(x)$ equals 
\begin{align}
\!\!\!\!\!\!\!\!\!\!\!\!\!\!\!\!\!\!\!\!\!
\frac12\int_{M^2}\Bigg\{&
\Big(\Big|\He V(e,e)+\frac{1}{2}\mathrm{Ric}(e,e) V\Big|^2+2\big\langle \Delta^\perp V, \ca(V)\big\rangle-\frac{1}{2}\big\langle R_{e\bar e}^{\perp}V\wedge V, B(e,e)\wedge B(\bar e,\bar e)\big\rangle \nonumber\\
&~~~~~~~~~-\frac{1}{2}\bll \nabla_e^{\perp}V\otimes B(\bar e,\bar e),  \nb V\otimes\B\Brr +\bl \ric(V), \ca(V)\Br+\bl\ca(V),\ca(V)\Br
\Big) \nonumber\\
&+2
\mathfrak{Re}
\Big(
\frac{1}{2}\bl \He V(\Be,\Be),V\Br \bl\B,\vH\Br+\bl \He V(\Be,\Be),\vH\Br \bl\B,V\Br \nonumber\\
&~~~~~~~~~+W\big(e,V, e, \He V(\Be,\Be)+\frac{1}{2}\mathrm{Ric}(\be,\be)V+2\bl  B(\be,\be),V\Br\vH-\frac{1}{2}\bl  B(\be,\be), \vH\Br V\big) \nonumber\\
&~~~~~~~~~+\bl \nb V\otimes \B,\frac{1}{2}Ric(\Be)^\perp\otimes V-\frac{1}{2}\n\Bb\otimes V+\frac{1}{4}V\otimes \n \Bb\Br \nonumber\\
&~~~~~~~~~+\bll \nb V\otimes \B, \frac{1}{4}Ric(\Be)^\perp\otimes V+\frac{3}{4}V\otimes W^\perp(e,\Be,\Be)+\frac{1}{2}\vH\otimes \nb V\Brr \nonumber\\
&~~~~~~~~~-\frac{1}{2}\bll \nb \B\otimes \nb V, \vH\otimes V\Brr+\frac{1}{2}\bl W^\perp(e,\Be,\Be)\wedge V, V\wedge\nb\B\Br \label{eq-ori2ndvar}\\
&~~~~~~~~~+\frac{1}{2}\left|V\right|^2\left(\mathrm{Ric}(e,e)\bl \Bb,\vH\Br - W\big(\Be, \vH,\Be, \B\big)\right)-\frac{1}{4}\Big| V\wedge \nb B(e,e)\Big|^2 \nonumber\\
&~~~~~~~~~+\frac{1}{4}\left|  W(e,\bn, e,V)\right|^2+\frac{1}{4}\left|  W(e, n, e,V)\right|^2+\frac{1}{8}\left|  \mathrm{Ric}(e,e) V\right|^2 \nonumber\\
&~~~~~~~~~+\frac{1}{2}\left|  B(e,e) \right|^2\Big|\vH\wedge V\Big|^2+{3}\bl \vH,V\Br\bl \ca(V),\vH\Br+\nabla_V K\big(\Be, V,\Be,\B\big) \nonumber
\Big)\Bigg\}dA_x. 
\end{align}
\end{theorem}
\begin{remark}\label{rk-2ndform}
We can also use the following Ricci identities  
\begin{align}
&\vab-\vBa=-\frac{R_{1234}v}{2},\label{eq-vRicci1}\\
&\vaba-\vaab=\frac{1}{2}\left(R_{1234}-R_{1212}\right)v_1,\label{eq-vRicci2}\\
&\vbba-\vbab=\frac{1}{2}\left(R_{1234}+R_{1212}\right)v_{\bar1} \label{eq-vRicci3}
\end{align}
to simplify some terms in \eqref{eq-original} as below  (comparing with \eqref{eq-HessVN}, \eqref{eq-LplV} and \eqref{eq-vb1psi})
\begin{equation}\begin{split}
&\left| v_{11}\right|^2+\left | v_{\bar 1\bar 1}\right |^2-2\left(\left| v_{\bar{1}}\right| ^2 \left| \psi \right| ^2+\left| v_1\right| ^2 \left| \phi \right| ^2\right)~~~~~~~\\
&~=\frac{1}{8}\left|\Delta^\perp V\right|^2-\frac{1}{4}K_{1212}|\nb V|^2+K_{1234}\left(|v_1|^2-|v_{\bar 1}|^2\right)+\frac{1}{8}R_{1234}^2|V|^2~~~~~~~~~\\
&~=\frac{1}{8}\left|\Delta^\perp V\right|^2-\frac{1}{4}K_{1212}|\nb V|^2+\frac{1}{8}W(e,\Be,\nb V, \n V)+\frac{1}{8}R_{1234}^2|V|^2,~~~~~~~~~~~~~~~
\end{split}\end{equation}
\begin{equation}\begin{split}
&\mathfrak{Re}\Big[v_{1\bar{1}} \left(\frac{1}{2} \bar{v} \left(\left| \psi \right| ^2+3 \left| \phi \right| ^2\right)+2 v \psi  \bar{\phi }\right)
+v_{\bar{1}1} \left(\frac{1}{2} \bar{v} \left(\left| \phi \right| ^2+3 \left| \psi \right| ^2\right)+2 v \psi  \bar{\phi }\right)\Big]\\
&~=\frac{1}{2}\big\langle \Delta^\perp V, \ca(V)\big\rangle-\frac{1}{4}R_{1234}^2|V|^2+\frac{1}{4}K_{1234}^2|V|^2-\frac{1}{32}|V|^2W\left(e,\Be,\B,\Bb\right). 
\end{split}\end{equation}
Then the second variation formula of the Willmore functional $\mathcal{W}$ can be rewritten as follows, 
\begin{align}
\frac{\p^2}{\p t^2}\Big{|}_{t=0}\mathcal{W}(x)&=
\frac12\int_{M^2}\Bigg\{
\Big(
\left|\Delta^\perp V\right|^2-2K_{1212}|\nb V|^2+2\big\langle \Delta^\perp V, \ca(V)\big\rangle+W(e,\Be,\nb V, \n V) \nonumber \\
-\frac{|V|^2}{8}&W\left(e,\Be,\B,\Bb\right)+\bl \ric(V), \ca(V)\Br+\bl\ca(V),\ca(V)\Br+K_{1234}^2|V|^2 
\Big) \nonumber\\
 \!\!\!+2
\mathfrak{Re}\;&\Big(
\frac{1}{2}\bl \He V(\Be,\Be),V\Br \bl\B,\vH\Br+\bl \He V(\Be,\Be),\vH\Br \bl\B,V\Br \nonumber \\
&+W\big(e,V, e, \He V(\Be,\Be)+\frac{1}{2}\mathrm{Ric}(\be,\be)V+2\bl  B(\be,\be),V\Br\vH-\frac{1}{2}\bl  B(\be,\be), \vH\Br V\big) \nonumber \\
&+\frac{1}{2}\bl  \He V(\Be,\Be), V\Br \mathrm{Ric}(e,e)+\frac{1}{8}\left|  \mathrm{Ric}(e,e) V\right|^2+\nabla_V K\big(\Be, V,\Be,\B\big) \nonumber \\
&+\bl \nb V\otimes \B,\frac{1}{2}Ric(\Be)^\perp\otimes V-\frac{1}{2}\n\Bb\otimes V+\frac{1}{4}V\otimes \n \Bb\Br \label{eq-2nd2ndvar} \\
&+\bll \nb V\otimes \B, \frac{1}{4}Ric(\Be)^\perp\otimes V+\frac{3}{4}V\otimes W^\perp(e,\Be,\Be)+\frac{1}{2}\vH\otimes \nb V\Brr \nonumber \\
&-\frac{1}{2}\bll \nb \B\otimes \nb V, \vH\otimes V\Brr+\frac{1}{2}\bl W^\perp(e,\Be,\Be)\wedge V, V\wedge\nb\B\Br \nonumber\\
&+\frac{1}{2}\left|V\right|^2\left(\mathrm{Ric}(e,e)\bl \Bb,\vH\Br - W\big(\Be, \vH,\Be, \B\big)\right)-\frac{1}{4}\left| V\wedge \nb B(e,e)\right|^2 \nonumber \\
&+\frac{1}{4}\left|  W(e,\bn, e,V)\right|^2+\frac{1}{4}\left|  W(e, n, e,V)\right|^2+\frac{1}{2}\left|  B(e,e) \right|^2\left|\vH\wedge V\right|^2+{3}\bl \vH,V\Br\bl \ca(V),\vH\Br
\Big)\Bigg\}dA_x.  \nonumber
\end{align}

Especially, in $\mathbb{S}^4$, for closed minimal surfaces, we have 
\begin{equation*}
\begin{split}\frac{\p^2}{\p t^2}\Big{|}_{t=0}\mathcal{W}(x)=
\frac12\int_{M^2}
\Big(
\left|\Delta^\perp V\right|^2+2\bl\Delta^\perp V,V\Br+2\big\langle \Delta^\perp V, \ca(V)\big\rangle
+{2}\bl V, \ca(V)\Br+\bl\ca(V),\ca(V)\Br
\Big)dA_x, 
\end{split}
\end{equation*}
which is well-known (see \cite{Weiner,Ndiaye-Schatzle,Wang-Kusner}). 
\end{remark}

\section{The strictly Willmore stability of Clifford torus in $\mathbb{C}P^2$}\label{sec-torus}
In this section, we prove that the Clifford torus $T^2$ embedded in $\mathbb{C}P^2$ defined by \eqref{eq-T2} is strictly Willmore-stable, and also strictly stable for the functional $\mathcal{W}^-$, which strengthens the conjecture of Montiel and Urbano stated in the introduction. 

As pointed in \cite{Jensen-Liao}, the Clifford torus $T^2$ 
can be viewed as an isometric immersion of the flat torus 
{$$\mathbb{R}^2/\mathrm{Span}_\mathbb{Z}\{(\frac{\pi}{\sqrt 3}, -\frac{\pi}{3}),~(\frac{\pi}{\sqrt 3}, \frac{\pi}{3})\}$$}
into $\mathbb{C}P^2$ by 
$$x=\frac{1}{\sqrt{3}}[e^{2i s},~e^{i(\sqrt 3 t-s)},~e^{i(\sqrt 3 t+s)}].$$
The metric on this torus is given by 
$$2(dt^2+ds^2).$$
It is well known that  the horizontal lift of $x$ in $\mathbb{S}^5$ is also minimal, which is known as the equilateral torus. 

Define
$$e_1\triangleq \frac{1}{\sqrt{2}}x_*(\frac{\p}{\p t}),~~~e_2\triangleq \frac{1}{\sqrt{2}}x_*(\frac{\p}{\p s}),~~~e_3\triangleq -J(e_1),~~~e_4\triangleq J(e_2).$$
It is easy to verify that $\{e_1, e_2\}$ is the orthonormal tangent frame of $x$, while  $\{e_3, e_4\}$ is the  orthonormal normal frame.  In terms of the complex frame, we have 
\beq
J(e)=-\bar n,~~~J(\be)=-n,~~~J(n)=\be,~~~J(\bn)=e.
\eeq
We summary the basic geometric invariants of $x$ into the following lemma, which can be easily verified (for reference, see \cite{Ludden-Okamura-Yano}).  
\begin{lemma}
On the Clifford torus, we have 
\beq\label{eq-H11-torus}
H=0,~~~\psi=0,~~~\phi=\frac{\sqrt 2 i}{2},~~~\o_{12}=0,~~~\o_{34}=0,
\eeq
\beq\label{eq-psiphi-torus}
\phi_{1}=0,~~~\phi_{\bar 1}=0,~~~\psi_1=0,~~~\psi_{\bar 1}=0,~~~R_{1212}=0,~~~R_{1234}=0.
\eeq
\end{lemma}
Along $x$, by the following well-known formula 
\beq\label{eq-Kxyzw}
\begin{split}
K(X,Y,Z,W)=\bl X,Z\Br\bl Y,W\Br-\bl X,W\Br\bl Y,Z\Br+\bl X, J(Z)\Br\bl Y, J(W)\Br-\bl X, J(W)\Br\bl Y, J(Z)\Br\\
+2\bl X,J(Y)\Br\bl Z, J(W)\Br,
\end{split}
\eeq
we can see that the terms involving the ambient curvature tensor of $\mathbb{C}P^2$ in the second variation formula \eqref{eq-original} are all equal to zero, except 
\beq\label{eq-curvature-CP2-torus}
K_{1212}=1,~K_{1234}=-1,~K(e,\bn,\be,n)=8,~K(e,n,\be,\bn)=12,~K(e,n,e,n)=12,~K(\be,\bn,\be,\bn)=12.
\eeq

It follows from \eqref{eq-original}, \eqref{eq-psiphi-torus} and \eqref{eq-curvature-CP2-torus} that 
\beq\label{eq-2nd-torus}
\begin{split}
\frac{\p^2}{\p t^2}\Big{|}_{t=0}\mathcal{W}(x)&=2\int_{T^2}\Big(2\big(\left | v_{11}\right | ^2 +\left | v_{\bar 1\bar1}\right |^2 \big)+6|v|^2+ 3vv_{\bar1\bar 1}+3\bv\bv_{1 1}-2|v_{ 1}|^2+\bv v_{1\bar 1} +v \bv_{\bar 1 1}\Big)dA_x\\
&=2\int_{T^2}\Big(2\big(\left | v_{11}\right | ^2 +\left | v_{\bar 1\bar1}\right |^2 \big)+6|v|^2-3v_{\bar1}^2-3\bv_{1}^2-4|v_{1}|^2\Big)dA_x,
\end{split}
\eeq
where we have used the integration by parts, i.e., 
$$\int_{T^2} vv_{\bar1\bar 1}d A_x=-\int_{T^2} v_{\bar 1}^2 d A_x, ~~~\int_{T^2} \bv\,\bv_{11}d A_x=-\int_{T^2} \bv_{1}^2d A_x,$$
$$\int_{T^2} \bv v_{1\bar 1}d A_x=\int_{T^2} v \bv_{\bar 1 1}d A_x=-\int_{T^2}|v_{ 1}|^2d A_x.$$ 
Note that both the tangent bundle and normal bundle of $x$ are flat, therefore by  \eqref{eq-vRicci1} $\sim$ \eqref{eq-vRicci3} and \eqref{eq-psiphi-torus}, we have 
$$v_{1\bar 1}=v_{\bar1 1}, ~~~v_{11\bar 1}=v_{1\bar 1 1},~~~v_{\bar 1 \bar 1 1}=v_{\bar 1 1 \bar 1},$$
which further implies that 
$$\int_{T^2} |v_{1 1}|^2d A_x=\int_{T^2} |v_{\bar 1\bar 1}|^2 d A_x=\int_{T^2} |v_{ 1\bar 1}|^2 d A_x.$$
Substituting theses equations into \eqref{eq-2nd-torus}, we can derive that 
\beq\label{eq-2nd-torus1}
\begin{split}
\frac{\p^2}{\p t^2}\Big{|}_{t=0}\mathcal{W}(x)=2\int_{T^2}\Big(4\left | v_{1\bar1}\right | ^2 +6|v|^2-3v_{\bar1}^2-3\bv_{1}^2-4|v_{1}|^2\Big)dA_x.
\end{split}
\eeq
\begin{lemma}\label{lem-inequality}
For any given smooth complex-valued function $v=a+i b$ on $T^2$, where $a$ and  $b$ 
denote the real and imaginary part of $v$ respectively,  
along the normal variation determined by $V=\frac{v\bn+\bv n}{2}$, we have 
\beq\label{eq-2nd-torus2}
\begin{split}
\frac{\p^2}{\p t^2}\Big{|}_{t=0}\mathcal{W}(x)\geq \frac12\int_{T^2}\big(\left |  \triangle v \right | ^2+24|v|^2-10 |\nabla v|^2\big)dA_x,
\end{split}
\eeq
where 
$$\triangle v\triangleq\triangle a+i \triangle b,~~~\nabla v \triangleq \nabla a+i \nabla b.$$
Moreover, in \eqref{eq-2nd-torus2}, the equality holds if and only if 
$$i(a_{\bar{1}}-a_1)=b_1+b_{\bar{1}},~~i.e.,~~ -\frac{\partial a}{\partial s}=\frac{\partial b}{\partial t}.$$
\end{lemma}
\begin{proof} 
It is obvious that 
$$a_{\bar1} a_{1}=\frac{1}{4}|\nabla a|^2,~~~b_{\bar1} b_{1}=\frac{1}{4}|\nabla b|^2,~~~ v_{1\bar1} =v_{\bar1 1}=\frac{1}{4}\triangle v.$$

Using 
\beq\label{eq-a1bb1}
2i \int_{T^2} \big(a_{\bar1} b_{1}-b_{\bar1} a_{1}\big)d A_x=\int_{T^2} da\wedge db=0,
\eeq
we can derive that 
\beq \label{eq-ingradv}
\int_{T^2} |v_{1}|^2 d A_x=\int_{T^2} v_{1} \bv_{\bar1}d A_x=\int_{T^2} \big(a_{\bar1} a_{1}+b_{\bar1} b_{1}\big)d A_x=\frac{1}{4}\int_{T^2} |\nabla v |^2.
\eeq

Note that 
\begin{equation*}
\begin{split}
&v_{\bar 1}^2+\bv_{1}^2=a_1^2-b_1^2-2ia_1b_1+a_{\bar 1}^2-b_{\bar 1}^2+2ia_{\bar 1}b_{\bar 1}\\
=&a_1^2+a_{\bar 1}^2-(b_1^2+b_{\bar 1}^2)+2i(a_{\bar 1}-a_1)(b_1+b_{\bar 1})-2i(a_{\bar1} b_{1}-b_{\bar1} a_{1})\\
=&\frac{1}{2}\big((a_1-a_{\bar1})^2+(a_1+a_{\bar1})^2\big)-\frac{1}{2}\big((b_1-b_{\bar1})^2+(b_1+b_{\bar1})^2\big)+2i(a_{\bar 1}-a_1)(b_1+b_{\bar 1})-2i(a_{\bar1} b_{1}-b_{\bar1} a_{1})\\
\leq&\frac{1}{2}\big((a_1+a_{\bar1})^2-(a_1-a_{\bar1})^2\big)+\frac{1}{2}\big((b_1+b_{\bar1})^2-(b_1-b_{\bar1})^2\big)-2i(a_{\bar1} b_{1}-b_{\bar1} a_{1})\\
=&2(a_{\bar1} a_{1}+b_{\bar1} b_{1})-2i(a_{\bar1} b_{1}-b_{\bar1} a_{1}),
\end{split}
\end{equation*}
where we have used Cauchy-Schwartz inequality to $2i(a_{\bar 1}-a_1)(b_1+b_{\bar 1})$. It follows from \eqref{eq-a1bb1} and \eqref{eq-ingradv} that 
\beq \label{eq-ingradv2}
\int_{T^2} (v_{\bar 1}^2+\bv_{1}^2) d A_x\leq \frac{1}{2}\int_{T^2} |\nabla v |^2.
\eeq
Subitituting \eqref{eq-a1bb1} $\sim$ \eqref{eq-ingradv2} into \eqref{eq-2nd-torus2}, we complete the proof. 
\end{proof}
Next, we consider the following quadratic form on $C^{\infty}(T^2)$: 
$$Q(\alpha,\beta)\triangleq \int_{T^2}\big( \triangle \alpha \triangle \beta+24 \alpha \beta-10 \langle  \nabla \alpha,  \nabla \beta \rangle \big)dA_x,$$
and extend it in the natural way to a Hermitian  quadratic form on the space of smooth complex-valued functions on $T^2$. Then we can rewrite \eqref{eq-2nd-torus2} as follows, 
$$\frac{\p^2}{\p t^2}\Big{|}_{t=0}\mathcal{W}(x)\geq \frac12Q(v,v).$$

\begin{theorem}
The Clifford torus in $\mathbb{C}P^2$ is strictly Willmore-stable. 
\end{theorem}
\begin{proof}
This theorem follows if we can prove that the quadratic form $Q$ on $C^{\infty}(T^2)$ is semi-positive definite. 

Given two distinct real-valued eigenspaces, we know they are orthogonal under the $L^2$ inner product. It is straightforward to verify that they are also orthogonal with respect to the quadratic form $Q(\cdot, \cdot)$. Therefore, we only need to verify that the restriction of $Q$ on every real-valued eigenspace is semi-positive definite. 

Suppose $f$ is an eigenfuction corresponding to the eigenvalue $\lambda_k$. Then 
$$Q(f,f)=\int_{T^2}\big(|\triangle f|^2 +24 f^2 -10  | \nabla f|^2\big)d A_x =\int_{T^2}\big(24+\lambda_k(\lambda_k-10)\big)f^2d A_x.$$

As a flat torus, it is straightforward to verify 
that the spectrum of $x$ is given by 
{
$$\{6m^2+6n^2+6mn \, |\, m, n\in \mathbb{Z}\},$$
and the first two nonzero eigenvalues of $x$ are 
\beq\label{eq-spec}
\lambda_1=6, ~~~\lambda_2= 18.
\eeq
Therefore, we have  $Q(f,f)=0$ when $k=1$, and 
$Q(f,f)>0$ when $k\geq2$. 
Therefore $Q$ is semi-positive definite.}  
Hence, we have  
\beq \label{eq-WQ}
\frac{\p^2}{\p t^2}\Big{|}_{t=0}\mathcal{W}(x)\geq \frac12Q(v,v)=Q(a,a)+Q(b,b)\geq 0,
\eeq
where $a$ and $b$ are two real functions on $T^2$ and $v=a+ib$. 

To establish the strict Willmore stability of \( x \), we rely on Lemma~\ref{lem-inequality}. According to this lemma, the equality in \eqref{eq-WQ} holds if and only if the condition
\[
-\frac{\partial a}{\partial s} = \frac{\partial b}{\partial t}
\]
is satisfied. Here, both functions \( a \) and \( b \) must reside within the first eigenspace associated with \( x \). This eigenspace is spanned by the set of functions:
\[
\left\{ \cos(-\sqrt{3}t + 3s), ~ \cos(-\sqrt{3}t - 3s), ~ \cos(2\sqrt{3}t), ~ \sin(-\sqrt{3}t + 3s), ~ \sin(-\sqrt{3}t - 3s), ~ \sin(2\sqrt{3}t) \right\}.
\] 
It is straightforward to see that such kind of $\{a, b\}$ must take the form 
\begin{equation*}
 \begin{split}
&a=a_1 \cos(-\sqrt{3}t+3s)+a_2 \cos(-\sqrt{3}t-3s)+a_3 \sin(-\sqrt{3}t+3s)+a_4 \sin(-\sqrt{3}t-3s)\\
&~~~~~+a_5 \cos(2\sqrt{3}t)+a_6\sin(2\sqrt{3}t), \\
&b=\sqrt{3}a_1 \cos(-\sqrt{3}t+3s)-\sqrt{3}a_2 \cos(-\sqrt{3}t-3s)+\sqrt{3}a_3 \sin(-\sqrt{3}t+3s)-\sqrt{3}a_4 \sin(-\sqrt{3}t-3s),
\end{split}
\end{equation*}
where $a_1, \cdots, a_6$ are real constants. 
It follows that nullity of the Willmore functional $\mathcal{W}$ at $x$  is $6$, which is exactly the dimension of isometric (and consequently, conformal) transformations of $\mathbb{C}P^2$ that do not preserve the Clifford torus. Thus, $x$ is strictly Willmore-stable.  
\end{proof}
{
\begin{remark}
We can also prove that the Clifford torus is stable for the functional $\mathcal{W}^-$. In fact, similar calculation gives us that 
\begin{equation*}
\begin{split}
\frac{\p^2}{\p t^2}\Big{|}_{t=0}\mathcal{W}^-(x)=\frac12\int_{T^2}\Big(16\left | v_{1\bar1}\right | ^2 +12|v|^2-32|v_{1}|^2\Big)dA_x=\frac12\int_{T^2}\big(|\triangle v|^2 +12 |v|^2 -8  | \nabla v|^2\big)d A_x, 
\end{split}
\end{equation*}
and the conclusion follows from the spectrum property of Clifford torus given in \eqref{eq-spec}. 
\end{remark}
\begin{remark}
Similarly, for every closed minimal Lagrangian surface in $\mathbb{C}P^2$, we have  
\beq\label{eq-WL}
\left.\frac{\partial^2}{\partial t^2}\right|_{t=0}\mathcal{W}(x) \geq \frac12 \int_{M^2} \left( |\mathcal{L}(V)|^2 - 2\langle \mathcal{L}(V), V\rangle \right) dA_x,
\eeq
\[
\left.\frac{\partial^2}{\partial t^2}\right|_{t=0}\mathcal{W}^-(x) = \frac12 \int_{M^2} \left( |\mathcal{L}(V)|^2 - 4\langle \mathcal{L}(V), V\rangle \right) dA_x,
\]
where $\mathcal{L}$ denotes the Jacobi operator of this surface for the area functional. Moreover, equality holds in \eqref{eq-WL} for Lagrangian variations. 
\end{remark}
\begin{remark} \label{rk-complex}
As shown in the introduction, complex curves in $\mathbb{C}P^2$ minimizes the Willmore functional $\mathcal{W}$. An application of our second variational formula to complex curves is that we can provide  a lower bound to the first nonzero eigenvalue of Jacobi operator.  Given a complex curve $x: M^2\rightarrow \mathbb{C}P^2$, by a straightforward calculation, we obtain that 
\begin{equation*}
\begin{split}
\frac{\p^2}{\p t^2}\Big{|}_{t=0}\mathcal{W^+}(x)&=8\int_{M^2}|v_{\bar 1\bar1}|^2 dA_x=\int_{M^2}\big( 8|v_{\bar 11}|^2 -24|v_{\bar1}|^2\big)dA_x,\\
&=\frac12\int_{M^2}\big( |\mathcal{L}(V)|^2+12\langle \mathcal{L}(V),V\rangle \big)dA_x,
\end{split}
\end{equation*}
where $\mathcal{L}\triangleq \triangle^\perp +2(1+|\phi|^2)$ is the Jacobi operator of $x$, and can be expressed as follows 
$$\mathcal{L}(V)=2(v_{1\bar1}\bn+\bv_{\bar1 1}n).$$  
Assume that $\Lambda_1$ is the first nonzero eigenvalue of the Jacobi operator,  then $\Lambda_1>0$. 
It  follows from 
$$\int_{M^2}\Lambda_1(\Lambda_1-12)dA_x=\frac{\p^2}{\p t^2}\Big{|}_{t=0}\mathcal{W^+}(x)\geq 0$$ 
that $\Lambda_1\geq 12$.  We point out that the equality can be attained by the Veronese curve $[1,\sqrt{2}z, z^2]$. In fact, for this curve, all spectrum of the Jacobi operator can be calculated by using the method of Montiel and Urbano developed in \cite{Montiel-Urbano2}, as well as the fact that the Veronese curve is of constant curvature and has parallel second fundamental form. 
\end{remark}
}

{\bf Acknowledgement:} We are supported by NSFC No. 12171473 and No. 11831005. The second author is grateful to Prof. P. Wang and Y. Lv for valuable discussions. 

{\bf Conflict of interest:} On behalf of all authors, the corresponding author states that there is no conflict of
interest. 
\par
~~~\\
\noindent{Changping Wang}\\
\noindent{\small \em School of Mathematics and Computer Science, FJKLMAA, Fujian Normal University, Fuzhou 350117, P. R. China.}\\
\noindent{\em Email: cpwang@fjnu.edu.cn}\\
~~~\\

\noindent{Zhenxiao Xie}\\
\noindent{\small \em School of Mathematical Sciences, Beihang University, Beijing 100191, P. R. China.}\\
\noindent{\em Email: 
xiezhenxiao@buaa.edu.cn}

\end{document}